\documentclass[11pt]{elsarticle}

\usepackage{amssymb}
\usepackage[boxed, linesnumbered, ruled]{algorithm2e}
\usepackage{amsfonts}
\usepackage{graphicx}
   \usepackage{amsmath, amssymb, amsthm}
\usepackage{mathrsfs}

\usepackage{tikz}
 \newtheorem{thm}{Theorem}[section]
 \newtheorem{cor}[thm]{Corollary}
 \newtheorem{lem}[thm]{Lemma}
 \newtheorem{rmk}[thm]{Remark}
 \newtheorem{prop}[thm]{Proposition}
 \newtheorem{obs}[thm]{Observation}
 \theoremstyle{definition}
 \newtheorem{defn}[thm]{Definition}
 \newtheorem{exmp}[thm]{Example}

\numberwithin{equation}{section}

 \newcommand{\set}[1]{\left\{#1\right\}}

\usepackage{atbegshi}
\AtBeginDocument{\AtBeginShipoutNext{\AtBeginShipoutDiscard}}

\makeatletter
\def\ps@pprintTitle{%
 \let\@oddhead\@empty
 \let\@evenhead\@empty
 \def\@evenfoot{}%
 \let\@evenfoot\@oddfoot }
\makeatother

\begin{document}

\begin{frontmatter}



\title{On characterizing proper max-point-tolerance graphs}
\author{Sanchita Paul}
\address{Department of Mathematics, Jadavpur University, Kolkata - 700 032, India.}
\ead{sanchitajumath@gmail.com}
\cortext[]{corresponding author}

\date{}

\begin{titlepage}


\null
\vspace{-4.5em}


\begin{abstract}
\noindent
{\footnotesize Max-point-tolerance graphs (MPTG) was introduced by Catanzaro et al. in 2017 as a generalization of interval graphs. This graph class has many practical applications in study of human genome as well as in signal processing for networks. The same class of graphs were also studied in the name of $p$-BOX($1$) graphs by Soto and Caro in 2015. In our article, we consider a natural subclass of max-point-tolerance graphs namely, {\em proper max-point-tolerance graphs} (proper MPTG) where intervals associated to the vertices are not contained in each other properly. We present the first characterization theorem of this graph class by defining certain linear ordering on the vertex set. In course of this study we prove proper max-point-tolerance graphs are asteroidal triple free and perfect. We also find proper max-point-tolerance graphs are equivalent to unit max-point-tolerance graphs. Further we note that MPTG (proper MPTG) and max-tolerance graphs (proper max-tolerance graphs) are incomparable. In conclusion we demonstrate relations between proper MPTG with other variants of MPTG and max-tolerance graphs.
}
\end{abstract}

\vspace{0.5em}\noindent
\end{titlepage}

\end{frontmatter}

\section{Introduction}

\noindent A graph $G=(V,E)$ is an \textit{interval graph} if each vertex is mapped to an interval on real line in such a way so that any two of the vertices are adjacent if and only if their corresponding intervals intersect. 
Many people \cite{Fish,Fulkerson} have done the
extensive research work on interval graphs for its world-wide practical applications in Computer Science and Discrete Mathematics. 
In 1962 Lekkerkerker and Boland \cite{Lekker} first characterize this graph class by proving them chordal and asteroidal triple free.
There after many combinatorial problems like finding maximal clique, independent set, coloring problem, domination and most importantly recognition has been solved in linear time for interval graphs.

\vspace{0.2em}  

\noindent Due to its significant applications in many theoritical and practical situations, interval graphs was generalized to several variations \cite{BDGS, Sg, Sen}. One of the most relevant variation of this graph class related to this article is \textit{tolerance graph} which was first introduced by Golumbic and Monma \cite{Golumbic1} in 1982.
A simple undirected graph $G=(V,E)$ is a {\em min-tolerance graph} (typically known as \textit{tolerance graphs}) if each vertex $u\in V$ corresponds to a real interval $I_u$ and a positive real number $t_u$, called {\em tolerance}, such that $uv$ is an edge of $G$ if and only if $|I_u\cap I_v|\geq \min \set{t_u,t_v}$. Golumbic and Trenk \cite{Golumbic2} introduced \textit{max-tolerance graphs} where each vertex $u\in V$ corresponds to a real interval $I_u$ and a  positive real number $t_u$ (known as tolerance) such that $uv$ is an edge of $G$ if and only if $|I_u\cap I_v|\geqslant \max\set{t_u,t_v}$. For max-tolerance graphs, we may assume $t_{u}\leq |I_{u}|$ for each $u\in V$ since otherwise $u$ becomes isolated. Max-tolerance graphs having a  representation in which no interval properly contains another are called {\em proper max-tolerance graphs}. Some combinatorial problems such as finding maximal cliques were obtained in polynomial time whereas the recognition problem was proved to be NP-hard \cite{Kaufmann} for max-tolerance graphs in 2006. Also a geometrical connection of max-tolerance graphs to semi-squares was found by Kaufmann et al. \cite{Kaufmann}. The research work by Golumbic and Trenk \cite{Golumbic2} can be referred to for more information on tolerance graphs. 

\vspace{0.2em}
\noindent In 2015 Soto and Caro \cite{Soto} introduced a new graph class, namely {\em $p$-BOX} graphs where each vertex corresponds to a box and a point inside it in the Euclidean $d$-dimensional space. Any two vertices are adjacent if and only if each of the corresponding points
are in the intersection of their respective boxes. When the dimension is one this graph class is denoted by $p$-BOX($1$). In 2017 Catanzaro et al.~\cite{Catanzaro} studied this dimension one graphs independently
by giving it a different name, {\em max-point-tolerance graphs} (MPTG)
where each vertex $u\in V$ corresponds to a pair of an interval and a point $(I_u,p_u)$, where $I_u$ is an interval on the real line and $p_u\in I_u$, such that $uv$ is an edge of $G$ if and only if $\set{p_u,p_v}\subseteq I_u\cap I_v$. More precisely each pair of intervals can ``tolerate'' a non-empty intersection without creating an edge until 
at least one distinguished point does not fall in 
the intersection. Then $G$ is said to be {\em represented} by $\{(I_{v},p_{v})|v\in V\}$.
They characterize MPTG by defining a special linear ordering to its vertex set. The graphs MPTG have a number of practical applications in human genome studies for DNA scheduling and in modelling of telecommunication networks for sending and receiving messages \cite{Catanzaro}. Recently in \cite{san} 
{\em central-max-point-tolerance graphs} (central MPTG) has been studied 
 considering $p_{u}$ as the center point of $I_{u}$ for each $u\in V$.
In course of this study the class of central MPTG graphs are proven to be same as {\em unit max-tolerance graphs} where each interval possesses unit length. 
\vspace{0.2em}

\noindent A natural and well studied subclass of interval graphs is the class of \textit{proper interval graphs} where no interval contains other properly. This graph class has various characterizations in terms of linear ordering of its vertices, consecutive ones's property of its associated augmented adjacency matrix, forbidden graph structure ($K_{1,3}$) etc \cite{G}. It is known \cite{Bogart1} that proper interval graphs are same as \textit{unit interval graphs} where each interval possesses same length. So the natural question arises that when the intervals are proper, what will be the characterization of MPTG.

\vspace{0.2em}
\noindent In this article we introduce {\em proper-max-point-tolerance graphs} (proper MPTG). It is a MPTG having a representation where no interval is properly contained in other. We find this graph class to be asteroidal triple free and perfect.
We obtain the first characterization theorem of this graph class by
introducing certain linear ordering on its vertex set, which can be an independent interest of study for the class of proper MPTG graphs. 
Interestingly we note that interval graphs form a strict subclass of proper MPTG. Analogous to unit interval graphs we define 
{\em unit-max-point-tolerance graphs} (unit MPTG) as an MPTG where all the intervals have equal length. Next we show proper MPTG graphs are same as unit MPTG. We investigate the connection between 
max-tolerance graphs (proper max-tolerance graphs) and MPTG (proper MPTG) and able to prove that these graph classes as incomparable. Incidentally we settle a query raised in the book \cite{Golumbic2} by Golumbic and Trenk that whether unit max-tolerance graphs are same as proper max-tolerance graphs or not.
In Conclusion section we show the relations between the subclasses of MPTG and max-tolerance graphs related to proper MPTG.

\section{Preliminaries}
\noindent A matrix whose entries are only zeros and ones is a {\em binary} matrix. For a simple undirected graph $G=(V,E)$ we call a binary matrix as the {\em augmented adjacency matrix} of $G$ if we replace all principal diagonal elements of the adjacency matrix of $G$ by one \cite{G}.

\vspace{0.3em}
\noindent The following characterization of MPTG is known:

\begin{thm} \label{mptg1} \cite{Catanzaro, san}
Let $G=(V,E)$ be a simple undirected graph. Then the following are equivalent.
\begin{enumerate}
\item $G$ is a MPTG.

\item There is an ordering of vertices such that for every quadruplet $x,u,v,y$ the following condition holds:\label{4pt}
\begin{equation}\label{4p1}
\text{If}\hspace{0.5em} x<u<v<y,\ xv,uy\in E \Longrightarrow uv\in E. \hspace{1em} (4\operatorname{-}point \hspace{0.3em} condition)
\end{equation}

\item  There is an ordering of vertices of $G$ such that for any $u<v$, $u,v\in V$,
\begin{equation}
uv\not\in E \Longrightarrow uw\not\in E\text{ for all }w>v\text{ or, }wv\not\in E\text{ for all }w<u.
\end{equation}

\item There exists an ordering of vertices such that every $0$ above the principal diagonal of the augmented adjacency matrix $A^{*}(G)$ has either all entries right to it are $0$ (right open) or, all entries above it are $0$ (up open).
\end{enumerate}
\end{thm}

\noindent If the vertices of $G$ satisfy any of the aforementioned conditions with respect to a vertex ordering, then we call the ordering as {\em MPTG ordering} of $G$. One can verify that MPTG orderings need not be unique for  $G$.

\vspace{0.2em}
\noindent From \cite{san} one can note that MPTG and max-tolerance graphs are not same. In Theorem \ref{incom} we prove that these graph classes are incomparable. To obtain the proof we need certain property of proper max-tolerance graphs as described in Corollary \ref{k23pf}. 
\begin{obs}\label{pmtg}
Let $G=(V,E)$ be a proper max-tolerance graph. Then there exist a vertex ordering ($\prec^{*}$) of $V$ such that for every quadruplet $v_{1},v_{2},v_{3},v_{4}$ the following holds: 
\begin{eqnarray*}\label{pmt}
\mbox{If} \hspace{0.5em}  v_{1}\prec^{*} v_{2}\prec^{*} v_{3}\prec^{*} v_{4}, \hspace{0.5em} v_{1}v_{3},v_{2}v_{4}\in E\Rightarrow \\ v_{2}v_{3}\in E  \hspace{0.5em} \mbox{and}  \hspace{0.5em} (v_{1}v_{2}\in E  \hspace{0.5em}\mbox{or} \hspace{0.5em} v_{3}v_{4}\in E \hspace{0.5em} \mbox{or}\hspace{0.5em} v_{1}v_{2},v_{3}v_{4}\in E). \hspace{ 1 em} (2.3)
\end{eqnarray*}
\end{obs}

\begin{proof}
Let $G$ be a proper max-tolerance graph with interval $I_{i}=[a_{i},b_{i}]$ and tolerance $t_{i}$ for each vertex $v_{i}\in V$. Then arranging the intervals according to the increasing order of their left endpoints one can show $b_{2}-a_{3}=b_{2}-a_{4}+a_{4}-a_{3}>t_{2}$ as $a_{4}>a_{3}$ and $v_{2}v_{4}\in E$. Again $b_{2}-a_{3}=b_{2}-b_{1}+b_{1}-a_{3}>t_{3}$ as $b_{2}>b_{1}$ and $v_{1}v_{3}\in E$. Therefore we get $|I_{2}\cap I_{3}|=b_{2}-a_{3}> \text{max} \{t_{2},t_{3}\}$ which imply $v_{2}v_{3}\in E$.

\noindent Now if $v_{1}v_{2}\notin E$ then $b_{1}-a_{2}<t_{1}$ or $t_{2}$. If $b_{1}-a_{2}<t_{1}$ then $t_{1}>b_{1}-a_{2}>b_{1}-a_{3}\geq t_{1}$  introduces contradiction. Hence $b_{1}-a_{2}<t_{2}$. Thus we get $t_{3}\leq b_{1}-a_{3}<b_{1}-a_{2}<t_{2}$. Again if $v_{3}v_{4}\notin E$ one can find $t_{2}<t_{3}$ which is again a contradiction. Hence the result follows.
\end{proof}
 
\begin{cor}\label{k23pf}
$K_{2,3}$ is not a proper max-tolerance graph.
\end{cor}

\begin{proof}
On the contrary, if $K_{2,3}$ is a proper max-tolerance graph then one can find a vertex ordering ($\prec^{*}$ say) from Observation \ref{pmtg} with respect to which the vertices of $K_{2,3}$ satisfies (\ref{pmt}). Let $\{v_{1},v_{5}\},\{v_{2},v_{3},v_{4}\}$ are two partite sets of it.
Now if $v_{1}\prec^{*} v_{5}$. Then adjacency of vertices helps us to conclude that any two vertices from $\{v_{2},v_{3},v_{4}\}$ can not occur simultaneously within $v_{1}, v_{5}$ or left to $v_{1}$ or right to $v_{5}$ in $\prec^{*}$. Hence exactly one among them can occur within $v_{1}, v_{5}$ and between two other vertices one (say $v_{4}$) occurs left to $v_{1}$ and other (say $v_{3}$) occurs right to $v_{5}$ respectively. But again (\ref{pmt}) gets contradicted for vertices $\{ v_{4}, v_{1}, v_{5}, v_{3}\}$ as $v_{1},v_{5}$ are nonadjacent. Similar contradiction will arise when $v_{5}\prec^{*} v_{1}$. 
\end{proof}

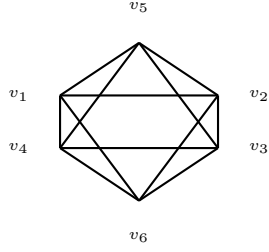
\begin{figure}[t]
\begin{center}
\begin{tikzpicture} [scale=0.7]
\draw[-][draw=black,thick] (1,0) -- (4,0);
\draw[-][draw=black,thick] (1,-1) -- (4,-1);
\draw[-][draw=black,thick] (1,0) -- (1,-1);
\draw[-][draw=black,thick] (4,0) -- (4,-1);
\draw[-][draw=black,thick] (2.5,1) -- (1,0);
\draw[-][draw=black,thick] (2.5,1) -- (4,0);
\draw[-][draw=black,thick] (2.5,-2) -- (1,-1);
\draw[-][draw=black,thick] (2.5,-2) -- (4,-1);
\draw[-][draw=black,thick] (2.5,1) -- (1,-1);
\draw[-][draw=black,thick] (2.5,1) -- (4,-1);
\draw[-][draw=black,thick] (2.5,-2) -- (1,0);
\draw[-][draw=black,thick] (2.5,-2) -- (4,0);
\node[above] at (2.5,1.4){\tiny{$v_{5}$}};
\node[below] at (2.5,-2.4){\tiny{$v_{6}$}};
\node[left] at (0.6,0) {\tiny{$v_{1}$}};
\node[right] at (4.4,0){\tiny{$v_{2}$}};
\node[left] at (0.6,-1) {\tiny{$v_{4}$}};
\node[right] at (4.4,-1) {\tiny{$v_{3}$}};
\end{tikzpicture}
\caption{$G_{2}$ in Example $2.6$ of \cite{san}}\label{g2}
\end{center}
\end{figure}

\begin{thm} \label{incom}
Max-tolerance graphs and the class of max-point-tolerance graphs are not comparable.
\end{thm}

\begin{proof}
From \cite{san} one can verify that the graph $G_{2}$ in Example $2.6$ (see Figure \ref{g2}) is a max-tolerance graph which is not a MPTG. Now in Lemma \ref{notMTG} we  prove $K_{m,n}$, $m,n\geq 13$ is not a max-tolerance graph. Next in Lemma \ref{MPTG} we show that $K_{m,n}$ for any two positive integers $m,n$ is a max-point-tolerance graph. Thus $K_{m,n}$ for $m,n\geq 13$ and $G_{2}$ seperate MPTG from max-tolerance graphs. Hence these graph classes become incomparable. 
\end{proof}

\noindent We do the proof verification of the aforementioned Theorem with the help of following Lemmas. 

\begin{lem}\label{notMTG}
Complete bipartite graph $K_{m,n}$ when $m,n\geq13$ is not a max-tolerance graph. 
\end{lem}

\begin{proof}
We assume on contrary $K_{m,n}$ for $m,n\geq 13$ is a max-tolerance graph with interval representation $\{I_{v_{i}}=[a_{v_{i}},b_{v_{i}}]|v_{i}\in V\}$ and tolerance $t_{v_{i}}$ for each vertex $v_{i}\in V$ where the vertex set $V$ is partitioned into two partite sets $V_{1}=\{x_{1},\hdots, x_{m}\}$, $V_{2}=\{y_{1},\hdots, y_{n}\}$ and $E$ be the edge set of $K_{m,n}$.
\vspace{0.3em}

\noindent \textbf{Claim $1$} \hspace*{0.1em} \textbf{No two intervals from same partite set can contain\footnote{$I_{x}$ contains $I_{y}$ means $I_{x}\supseteq I_{y}$} an interval from other partite set.}

\noindent We assume on contrary $I_{x_{1}}, I_{x_{2}}$ contain $I_{y_{1}}$. As $x_{1}y_{1}, x_{2}y_{1}\in E$, $|I_{y_{1}}|=|I_{x_{1}}\cap I_{y_{1}}|=|I_{x_{2}}\cap I_{y_{1}}|\geq t_{x_{1}}, t_{x_{2}}$ from definition. Now if $|I_{x_{1}}\cap I_{x_{2}}|<t_{x_{1}}$ then $t_{x_{1}}\leq |I_{y_{1}}|\leq |I_{x_{1}}\cap I_{x_{2}}|<t_{x_{1}}$ as $I_{x_{1}}, I_{x_{2}}\supseteq I_{y_{1}}$. Similar contradiction will arise if $|I_{x_{1}}\cap I_{x_{2}}|<t_{x_{2}}$.

\vspace{0.3em}
\noindent \textbf{Claim $2$} \hspace*{0.1em}  \textbf{There always exist three vertices from each partite set for which their corresponding intervals are not contained in each other.} (we only need $m,n\geq7$)

\noindent Note that $V_{1}, V_{2}$ are posets with respect to the interval containment relation. Hence if our claim is not true, as $m,n\geq 7$  each poset  must contain a totally ordered subset of cardinality at least 3. Let the totally ordered subset of $V_{1}$ be $\{I_{x_{1}}, I_{x_{2}}, I_{x_{3}}\}$ satisfying $I_{x_{3}}\subseteq I_{x_{2}}\subseteq I_{x_{1}}.$ 
  
\noindent We will show there can not be two intervals from $V_{2}$ which intersect $I_{x_{1}}$ in its same end (left or right). On contrary let $I_{y_{1}}$ and $I_{y_{2}}$ intersect $I_{x_{1}}$ in its left end and $b_{y_{1}}\leq b_{y_{2}}$. Now as $I_{x_{2}}$ contained in $I_{x_{1}}$, $|I_{y_{1}}\cap I_{y_{2}}|\geq |I_{y_{1}}\cap I_{x_{2}}|\geq t_{y_{1}}$. Again $t_{y_{2}}\leq |I_{y_{2}}\cap I_{x_{2}}|\leq |I_{x_{2}}|=|I_{x_{1}}\cap I_{x_{2}}|<t_{x_{1}}$\footnote{If $x,y$ are nonadjacent and $I_{y}\subseteq I_{x}$ then $|I_{x}\cap I_{y}|<t_{x}$}$\leq |I_{y_{1}}\cap I_{x_{1}}|\leq |I_{y_{1}}\cap I_{y_{2}}|$. Hence it follows that $|I_{y_{1}}\cap I_{y_{2}}|\geq t_{y_{1}}, t_{y_{2}}$ which is a contradiction as $y_{1},y_{2}$ are nonadjacent. 

\noindent Similarly one can show there can be at most one interval from $V_{2}$ which intersect $I_{x_{2}}$ in its same end. Now as $n\geq7$ there must exist atleast $n-2$ intervals corresponding to vertices of $V_{2}$ properly contained in $I_{x_{2}}$ which is not possible due to Claim $1$.

\vspace{0.3em}

\noindent  As $|V_{1}|\geq13(\geq7)$ we can always choose three vertices $A=\{x_{1},x_{2},x_{3}\}$ from $V_{1}$  satisfying the conditions of Claim $2$. Similarly we choose vertices $B=\{x_{4},x_{5},x_{6}\}$ from
$V_{1}\setminus A$ and $C=\{x_{7},x_{8},x_{9}\}$ from $(V_{1}\setminus A)\setminus B$ by observing $|V_{1}\setminus A|\geq 10$ and $|(V_{1}\setminus A)\setminus B|\geq 7$ respectively. Now as each $I_{x_{i}}$ for $1\leq i \leq9$ can be contained in atmost one $I_{y_{j}}$ from Claim $1$, we can discard atmost nine such $y_{j}$'s from $V_{2}$ where $1\leq j\leq n$. Again as $n\geq 13$ if we discard these $y_{j}$'s from $V_{2}$ we can again able to choose two vertices $y_{10},y_{11}$ (say) for which $I_{y_{10}}\not \subset I_{y_{11}}$ and $I_{y_{11}}\not \subset I_{y_{10}}$. If not then there will be atleast four $y_{j}$'s remaining in $V_{2}$  forming a totally ordered subset. As $m\geq 13$ it follows from the proof of claim $2$ that the above statement is redundant.
Clearly $I_{x_{i}}\not\subset I_{y_{10}}, I_{y_{11}}$ for $1\leq i\leq 9$. Again as $I_{y_{10}}$ and $I_{y_{11}}$ can contained in atmost one interval of $V_{1}$ from Claim $1$, between $A, B, C$ we can always get one set (say $C$) for which $I_{y_{10}}, I_{y_{11}}$ is not contained in any of the intervals of that set. Now it is easy to see $\{y_{10},y_{11}\}\cup C$ form $K_{2,3}$ for which no interval is contained in other. Therefore they form an induced proper max-tolerance graph. But it can not be true from Corollary \ref{k23pf}. 
\end{proof}

\begin{lem}\label{MPTG}
$K_{m,n}$ is max-point-tolerance graph for all positive integers $m,n$.
\end{lem}

\begin{proof}
Let $V_{1}=\{x_{1},\hdots,x_{m}\}, V_{2}=\{y_{1},\hdots,y_{n}\}$ be two partite sets of $K_{m,n}$. We assign intervals $I_{x_{i}}=[i,m+n+\dfrac{i\epsilon}{m}]$ and points $p_{i}=i$ for $1\leq i\leq m$ for  vertices of $V_{1}$ and $I_{y_{j}}=[\dfrac{j\epsilon}{n}, j+m]$ and points $p_{j}=j
+m$ for $1\leq j\leq n$ for vertices of $V_{2}$ where $0<\epsilon<1$. It is easy to check that the above assignment actually gives max-point-tolerance representation of $K_{m,n}$. 
\end{proof}

\section{Proper max-point-tolerance graphs (proper MPTG)}\label{pmpt}

\begin{prop}\label{distinct}
If a graph $G$ has a proper MPTG representation $S$, then $G$ has a proper MPTG representation $T$ in which the set of all points associated to each interval are distinct.
\end{prop}

\begin{proof}
Let $G=(V,E)$ be a proper MPTG and $S=\{(I_{v},p_{v})| v\in V\}$ be a proper MPTG representation of $G$ where $I_{v}=[a_{v},b_{v}]$  
be an interval and $p_{v}$ is a point within it. We arrange the intervals according to increasing order of $p_{v}$'s, $p_{1}$ through $p_{n}$.
We iterate the following rule at each shared point moving from left to right until all $p_{v}$-values become distinct. Suppose $p_{i}$ is the smallest $p_{v}$-value shared by multiple points in $S$. Then $p_{i}=p_{i+1}=\hdots=p_{m}$. It is easy to see there can be atmost two intervals having their endpoints as $p_{i}$ as the intervals are proper we can consider the interval endpoints to be distinct. Now if they are two in number then one of them must be left endpoint (say $a_{k_{1}}$) and another must be right end point (say $b_{k_{2}}$). Take a positive value of $\epsilon$ smaller than $\text{min}\{d(p_{m},p_{m+1}), d(p_{m},a_{k}), d(p_{m},b_{k^{'}})\}$ where $p_{m+1}$, $a_{k}$, $b_{k^{'}}$ be the associated point, left endpoint and right endpoint respectively occurred just after $p_{m}$ along the real line. To obtain a new proper MPTG representation $T$ of $G$ we assign $a_{k_{1}}=p_{i}$,
$p_{l}=p_{i}+\dfrac{\epsilon}{m-l+2}$ for all $i\leq l\leq m$, $b_{k_{2}}=p_{i}+\epsilon$. It is easy to check now that the adjacencies remain intact and the intervals remain proper in $T$. 
\end{proof}

\begin{prop}\label{complete}
Complete graphs $K_{n}$ for any positive integer $n$ are proper MPTG.
\end{prop}

\begin{proof} 
We give proper MPTG representation of $K_{n}$ by assigning intervals $I_{v_{k}}=[a_{k},b_{k}]$ and points $p_{k}$ for each vertex $v_{k}\in V$ where $1\leq k\leq n$. $I_{v_{1}}=[1+\dfrac{1}{n},n+1]$, $p_{1}=2$, $I_{v_{2}}=[1+\dfrac{2}{n},n+2]$, $p_{2}=2.5$ and $I_{v_{i}}=[1+\dfrac{i}{n},n+i]$, $p_{i}=i$ for $3\leq i\leq n$.
\end{proof}

\begin{prop}\label{bipt}
Complete bipartite graphs $K_{m,n}$, for any positive integer $m$ and $n$ are proper MPTG.
\end{prop}

\begin{proof}
Follows from the same MPTG representation of $K_{m,n}$ given in Lemma \ref{MPTG}.
\end{proof}

\noindent A graph $G$ is a {\em caterpillar} if it is a tree with path $s_{1},s_{2},\hdots,s_{k}$, called the spine of $G$ such that every vertex of $G$ has atmost distance one from the spine.

\begin{prop}
Caterpillars are proper MPTG.
\end{prop}

\begin{proof}
Let $G=(V,E)$ be a caterpillar with spine vertices $s_{1},s_{2},\hdots,s_{k}$ and additional vertices $a_{1},a_{2},\hdots,a_{l}$. Let $a_{k_{i-1}+1},\hdots,a_{k_{i}}$ be the vertices adjacent to spine vertex $s_{i}$ where $1\leq i\leq k$, $k_{0}=0$ and $1\leq k_{i}\leq l$. We set intervals $I_{s_{i}}=[2i-3,2i+2]$ and points $p_{s_{i}}=2i$ for spine vertices where $1\leq i\leq k$ and intervals $I_{a_{j}}=[2i-3-\dfrac{n-j+1}{n+1}, 2i+\dfrac{2j}{2n+1}]$, points $p_{a_{j}}=2i+\dfrac{2j-1}{2n+1}$ for additional vertices adjacent to spine vertex $s_{i}$ where $n=k_{i}-k_{i-1}$= number of vertices adjacent to $s_{i}$ and $1\leq j\leq n$.
\end{proof}

\noindent Let $N_{G}(x)$ $(N_{G}^{'}(x))$ be the set of all vertices which are adjacent (nonadjacent) to a vertex $x$ of a graph $G$. They are called {\em neighbours} ({\em non neighbours}) of $x$ in $G$.  The subgraph induced by neighbours of $x$ in $G$ is called {\em neighbourhood} of $x$ in $G$. An independent set of three vertices where each pair is joined by a path that avoids the neighborhood of the third is called an {\em asteroidal triple}. A graph is said to be {\em asteroidal triple free} (AT-free) if it does not contain any asteroidal triple. For vertices $u,v$ of a graph $G$, let $D(u,v)$ denote the set of vertices that intercepts (i.e; either the vertex is in the path or adjacent to some vertex of the path) all $u,v$ paths. For vertices $u,v,x$ of $G$, $u$ and $v$ are said to be {\em unrelated} to $x$ if $u\notin D(v,x)$ and $v\notin D(u,x)$ \cite {Corneil}. 

\begin{thm}\label{atfree}\cite{Corneil}
A graph $G$ is asteroidal triple free if and only if for every vertex $x$ of $G$, no component $F$ of $N_{G}^{'}(x)$ contains unrelated vertices.
\end{thm}

\noindent Lekkerkerker and Boland exhibited the impotance of asteroidal triples in \cite{Lekker} by proving interval graphs as chordal and AT-free. It was proved in \cite{Catanzaro} that interval graphs form a strict subclass of MPTG. But there are certain MPTG graphs which are not AT-free (see Figure \ref{atfree1}). But interestingly in the following Lemma we find the class of proper MPTG graphs to be asteroidal triple free. 

\begin{lem} \label{atfree2}
If $G$ is a proper MPTG then it is asteroidal triple free.
\end{lem}

\begin{proof}
Let $G=(V,E)$ be a proper MPTG with representation $\{(I_{v},p_{v})|v\in V\}$ where $I_{v}=[a_{v},b_{v}]$ be an interval and $p_{v}$ be a point within it. On contrary to Theorem \ref{atfree} lets assume there exist a vertex $x\in V$ and a component $F$ of $N_{G}^{'}(x)$ containing a pair of unrelated vertices $u,v$ (say). Let $P_{1},P_{2}$ be two arbitary paths of $D(x,v)$ and $D(x,u)$ respectively such that $u(v)$ is not adjacent to any vertex of $P_{1}(P_{2})$ where $P_{1}=(x,v_{1},\hdots,v_{n}=v), P_{2}=(x,u_{1},\hdots,u_{m}=u)$. 

\noindent Let $p_{u},p_{v}$ occur in same side of $p_{x}$ on real line (say $p_{u}<p_{v}<p_{x}$). Now if $p_{u_{1}}<p_{v}$ then $p_{u_{1}}<p_{v}<p_{x}$ clearly. Now as $u_{1},x$ are adjacent 
\begin{equation}\label{e}
a_{x}\leq p_{u_{1}}<p_{v}<p_{x}\leq b_{u_{1}}
\end{equation}
which imply $p_{v}\in I_{u_{1}}\cap I_{x}$. But as $v$ is not adjacent to $u_{1},x$, $p_{u_{1}}<a_{v}<b_{v}<p_{x}$ imply $I_{v}\subsetneq I_{u_{1}}, I_{x}$ from (\ref{e}), which a contradiction as the intervals are proper. Again if $p_{u_{m-1}}>p_{v}$ one can similarly find contradiction as $u_{m-1}u\in E$ and $vu,vu_{m-1}\notin E$. Hence $p_{u_{m-1}}(p_{u_{1}})$ must occur left (right) to $p_{v}$, i.e; $p_{u_{m-1}}<p_{v}<p_{u_{1}}$.
Now replacing $u$ by $u_{m-1}$ and $x$ by $u_{1}$ one can similarly find $p_{u_{m-2}}<p_{v}<p_{u_{2}}$. Hence using the above technique for finite number of times (as length of $P_{2}$ is finite) one can get $p_{u_{m-k}}<p_{v}<p_{u_{k}}$ for some $1<k<m$ such that $u_{k},u_{m-k}$ are adjacent in $P_{2}$. But this again introduce similar contradiction. 
 
\vspace{0.2em}
\noindent Now consider the case when $p_{u},p_{v}$ occur opposite to $p_{x}$ on real line (say $p_{u}<p_{x}<p_{v}$). As $F$ is a connected component of $N_{G}^{'}(x)$ there must exists a path $P=(u,\hdots,v)$ between $u,v$ in $F$. It is easy to observe that $p_{x}$ must fall within  an interval $I_{k}$ (say) for some vertex $k$ of the path $P$. Again as $k$ is not adjacent to $x$ either $p_{k}<a_{x}$ or $p_{k}>b_{x}$. Now if $p_{k}<a_{x}$ then there must exist some vertices $k^{'},k^{'}+1$ in $P$ such that $a_{k^{'}+1}\leq p_{k^{'}}<a_{x}<b_{x}<p_{k^{'}+1}\leq b_{k^{'}}$ where $k\leq k^{'}<v$ in $P$, which imply $I_{x}\subsetneq I_{k^{'}}, I_{k^{'}+1}$, which is not true as the intervals are proper. One can find similar contradiction when $p_{k}>b_{x}$. Hence we are done from Theorem \ref{atfree}.
\end{proof}

\noindent Hence from the above Lemma we can conclude now that all trees are not proper MPTG although they are all MPTG \cite{Soto}. It is not necessary that all AT-free graphs has to be perfect ($C_{5}$ is AT-free). However in Theorem \ref{perfect} we prove proper MPTG graphs to be perfect.  

\begin{lem} \label{Cn}
Let $C_{n}$ be a cycle of length $n$. Then $C_{n},n\geq 5$ does not belong to the class of proper MPTG.
\end{lem}

\begin{proof}
As proper MPTG graphs are AT-free follows from Lemma \ref{atfree2}, $C_{n},n\geq 6$ does not belong to the class of proper MPTG as $\{v_{1},v_{3},v_{n-1}\}$ form an asteroidal triple.
 \vspace{0.1em}
 
\noindent On contrary let $C_{5}=(V,E)$ has a proper MPTG representation $\{(I_{v},p_{v})|v\in V\}$ where $I_{v}=[a_{v},b_{v}]$ be an interval and $p_{v}$ be a point within it. Let $\{v_{i}|1\leq i\leq 5\}$ be the vertices occurred in circularly consecutive way in clockwise (anticlockwise) order in $C_{5}$. Now if $p_{1}$ occurs between $p_{2},p_{5}$ on real line (say $p_{2}<p_{1}<p_{5}$) then as $v_{2}v_{5}\notin E$ either $p_{2}<a_{5}$ or $p_{5}>b_{2}$.

\noindent If $p_{2}<a_{5}$ then $a_{1}<a_{5}$ (as $v_{1}v_{2}\in E$) and hence $b_{1}<b_{5}$ as the intervals are proper. Thus we get $I_{1}\cap I_{5}=[a_{5},b_{1}]$. Now as $v_{4}v_{5}\in E$, $p_{4}\in I_{5}$. Hence $p_{2}<a_{5}\leq p_{4}$. If $p_{4}$ occurs left to $p_{1}$ then $p_{4}\in [a_{5},p_{1}]$. But as $v_{1}v_{4}\notin E$, $b_{4}<p_{1}$ which imply $b_{4}<p_{5}$ which contradicts $v_{4}v_{5}\in E$. Hence $p_{4}>p_{1}$. Now as $v_{3}$ is adjacent to both $v_{2},v_{4}$, $[p_{2},p_{4}]\subseteq I_{3}$. Therefore we get $p_{1}\in I_{3}$ as $p_{2}<p_{1}<p_{4}$. But as $v_{1}v_{3}\notin E$, either $p_{3}>b_{1}$ or $p_{3}<a_{1}$.

\vspace{0.2em}
\noindent If $p_{3}>b_{1}$ then from the previous relations we get
\begin{equation}\label{p1}
a_{3}\leq p_{2}<a_{5}<p_{5}\leq b_{1}<p_{3}\leq b_{2}.
\end{equation}
Hence $p_{5}\in I_{3}$ from (\ref{p1}). But as $v_{3}v_{5}\notin E$, $b_{5}<p_{3}$, which imply $I_{5}\subsetneq I_{3}$ from (\ref{p1}), which is a contradiction as the intervals are proper. Again if $p_{3}<a_{1}$ then $a_{4}\leq p_{3}<a_{1}$ and therefore $b_{4}<b_{1}$. Combining we get 
  
  \begin{equation}\label{k1}
  a_{4}<a_{1}\leq p_{2}<a_{5}\leq p_{1}<p_{5}\leq b_{4}<b_{1}.
  \end{equation}
  
\noindent  Hence $p_{1}\in I_{4}$ from (\ref{k1}). But as $v_{1}v_{4}\notin E$, $p_{4}>b_{1}$. Thus we get $I_{1}\subsetneq I_{4}$ from (\ref{k1}), which is not true as the intervals are proper. One can similarly find contradiction when $p_{5}>b_{2}$.  

\noindent Now if $p_{2},p_{5}$ occurs in same side of $p_{1}$ (say $p_{2}<p_{5}<p_{1}$) then $p_{5}\in[p_{2},p_{1}]\subseteq I_{1},I_{2}$. But as $v_{2}v_{5}\notin E$, $a_{5}>p_{2}$ and hence $a_{5}>a_{1}$ and therefore 
$b_{5}>b_{1}$ as the intervals are proper. Hence 
\begin{equation}\label{s1}
I_{1}\cap I_{5}=[a_{5},b_{1}]. 
\end{equation}
Now $v_{4}v_{5}\in E$ imply $p_{4}\in I_{5}$. Now if $p_{4}<b_{1}$ then $p_{4}\in I_{1}\cap I_{5}$. But as $v_{1}v_{4}\notin E$, $b_{4}<p_{1}<b_{1}$ which imply $a_{4}<a_{1}$. Hence combining we get
$a_{4}<a_{1}\leq p_{2}<a_{5}\leq p_{5},p_{4}\leq b_{4}<p_{1}\leq b_{2}$ 
Thus $p_{2}\in I_{4}$ and $p_{4}\in I_{2}$, which is not true as $v_{2}v_{4}\notin E$. 

\vspace{0.2em}
\noindent Next if $p_{4}>b_{1}$ then $b_{4}\geq p_{4}>b_{1}$ imply $a_{4}>a_{1}$. Hence we get
\begin{equation}\label{s3}
a_{1}\leq p_{2}<a_{5}\leq p_{5}<p_{1}\leq b_{1}<p_{4}\leq b_{4},b_{5}.
\end{equation}
\noindent Now as $v_{3}$ is adjacent to $v_{2},v_{4}$, $[p_{2},p_{4}]\subseteq I_{3}$. This imply $p_{5},p_{1}\in I_{3}$ from (\ref{s3}). 
But as $v_{3}v_{5}, v_{1}v_{5}\notin E$, $p_{3}\notin I_{1}, I_{5}$ and therefore does not belong to $I_{1}\cap I_{5}$. Hence $p_{3}<a_{5}$ or $p_{3}>b_{1}$ from (\ref{s1}). 
\vspace{0.1em}

\noindent Now if $p_{3}<a_{5}$ then $p_{3}<a_{1}$ from (\ref{s3}) as $p_{3}\notin I_{1}$. Hence combining we get $a_{4}\leq p_{3}<a_{1}
<b_{1}<p_{4}\leq b_{4}$. This imply $I_{1}\subsetneq I_{4}$ which is a contradiction. Again if $p_{3}>b_{1}$ then $p_{3}>b_{5}$ as $p_{3}\notin I_{5}$ from (\ref{s3}). Now as $a_{3}\leq p_{2}$ from (\ref{s3}) we get $I_{5}\subsetneq I_{3}$ which is again not true as the intervals are proper. 
\end{proof}

\noindent We note that the $3$-wheel graph $W_{3}=K_{4}$ is a proper MPTG from Proposition \ref{complete}. Again the $4$-wheel graph $W_{4}$ is also a proper MPTG with representation
$([30,130],80)$ for center vetex $u$ and $([20,120],50),([10,100],70)$, $([60,150],90)$,
$([40,140],110)$ for other vertices.

\begin{cor}
The $n$-wheel graphs $W_{n}$ for $n\geq 5$ are not proper MPTG.
\end{cor}

\begin{proof}
If $W_{n},n\geq5$ be a proper MPTG then $W_{n}\setminus \{u\}$ would also be a proper MPTG where $u$ is the central vertex of $W_{n}$. But this is not true from Lemma \ref{Cn} as $W_{n}\setminus \{u\}=C_{n}$.
\end{proof}

\begin{lem}\label{Cnbar}
 $\overline{C_{n}},n\geq 7$ does not belong to the class of MPTG.
\end{lem}

\begin{proof}
One can note that $\{v_{1},v_{4},v_{n},v_{3}\}$ form an induced $4$-cycle ($C$ say) in $\overline{C_{n}},n\geq 7$. From (\ref{4p1}) one can find that in any MPTG ordering of vertices of $C$ neither $\{v_{1},v_{n}\}$ nor $\{v_{4},v_{3}\}$ can sit together within other.  
 
\noindent Now if we consider  $v_{1}\prec v_{3}\prec v_{n}\prec v_{4}$ in some MPTG ordering of $C$ then it is easy to verify from (\ref{4p1}) that when $n\geq8$ neither $v_{n-3}$ nor $v_{n-2}$ can occur prior to $v_{1}$ or after $v_{4}$ as $v_{n-3},v_{n-2}$ are adjacent to $v_{3}, v_{n}$ respectively. Hence they have to sit within $v_{1},v_{4}$. Now one can find contradiction from (\ref{4p1}) applying it on the vertex set $\{v_{1},v_{n-3},v_{n-2},v_{4}\}$ as $v_{n-3},v_{n-2}$ are nonadjacent. Similar contradiction will arise for other MPTG orderings of $C$.
Again for $n=7$ repeatedly applying (\ref{4p1}) one can able to show there exist no MPTG ordering of $\overline{C_{7}}$.  
\end{proof}

\noindent From the aforementioned Lemma it is obvious that proper MPTG graphs can not contain $\overline{C_{n}},n\geq 7$ as induced subgraph. Combining this fact with Lemma \ref{Cn} one can conclude the following

\begin{thm}\label{perfect}
The class of proper MPTG graphs are perfect.
\end{thm}

\begin{figure} 
\begin{minipage}[b]{0.5\textwidth}
\begin{tikzpicture}
\draw[-][draw=black, thick] (0,0) -- (0,0.5);
\draw[-][draw=black, thick] (0,0.5) -- (0,1);
\draw[-][draw=black, thick] (0,0) -- (0.5,-0.5);
\draw[-][draw=black, thick] (0,0) -- (1,-1);
\draw[-][draw=black, thick] (0,0)--(-0.5,-0.5);
\draw[-][draw=black, thick] (-0.5,-0.5)--(-1,-1);

\draw [fill=black] (0,0) circle [radius=0.09];
\draw [fill=black] (0,0.5) circle [radius=0.09];
\draw [fill=black] (0,1) circle [radius=0.09];
\draw [fill=black] (0.5,-0.5) circle [radius=0.09];
\draw [fill=black] (1,-1) circle [radius=0.09];
\draw [fill=black] (-0.5,-0.5) circle [radius=0.09];
\draw [fill=black] (-1,-1) circle [radius=0.09];

\node [below] at (0,0) {{$v_{5}$}};
\node [right] at (0.5,-0.5) {{$v_{7}$}};
\node [right] at (1,-1) {{$v_{3}$}};
\node [left] at (-0.5,-0.5) {{$v_{6}$}};
\node [left] at (-1,-1) {{$v_{2}$}};
\node [left] at (0,0.5) {{$v_{4}$}};
\node [left] at (0,1) {{$v_{1}$}};

\draw[-] [draw=black, thick] (0.4,1.1)--(3.7,1.1);
\draw[-] [draw=black, thick] (0.6,0.7)--(5,0.7);
\draw[-] [draw=black, thick] (1,0.3)--(6.3,0.3);
\draw[-] [draw=black, thick] (4.1,-0.1)--(7.7,-0.1);
\draw[-] [draw=black, thick] (2.5,-0.6)--(5.6,-0.6);
\draw[-] [draw=black, thick] (5.8,-0.6)--(8.1,-0.6);
\draw[-] [draw=black, thick] (3,-1)--(4.4,-1);

\draw [fill=black] (1,1.1) circle [radius=0.06];
\draw [fill=black] (2,0.7) circle [radius=0.06];
\draw [fill=black] (4.6,0.3) circle [radius=0.06];
\draw [fill=black] (5.1,-0.1) circle [radius=0.06];
\draw [fill=black] (3.4,-0.6) circle [radius=0.06];
\draw [fill=black] (7,-0.6) circle [radius=0.06];
\draw [fill=black] (3.8,-1) circle [radius=0.06];

\node [above] at (1,1.1) {{$p_{2}$}};
\node [above] at (2,0.7) {{$p_{6}$}};
\node [above] at (4.6,0.3) {{$p_{5}$}};
\node [above] at (5.1,-0.1) {{$p_{7}$}};
\node [above] at (3.4,-0.6) {{$p_{4}$}};
\node [above] at (7,-0.6) {{$p_{3}$}};
\node [above] at (3.8,-1) {{$p_{1}$}};

\node [right] at (3.7,1.1) {{$I_{2}$}};
\node [right] at (5.4,0.7) {{$I_{6}$}};
\node [right] at (6.3,0.3) {{$I_{5}$}};
\node [right] at (7.7,-0.1) {{$I_{7}$}};
\node [left] at (2.5,-0.6) {{$I_{4}$}};
\node [right] at (8.1,-0.6) {{$I_{3}$}};
\node [left] at (3,-1) {{$I_{1}$}};
\end{tikzpicture}
\caption{MPTG $G_{1}$ having asteroidal triple $\{v_{1},v_{2},v_{3}\}$}\label{atfree1}
\end{minipage}
\hfill
\begin{minipage}[b]{0.41\textwidth}

\begin{tikzpicture}
\draw[-][draw=black, very thick] (0,0) -- (2,0);
\draw[-][draw=black,very thick] (2,0) -- (2,2);
\draw[-][draw=black,very thick] (2,2) -- (0,2);
\draw[-][draw=black,very thick] (0,0) -- (0,2);
\draw [fill=black] (0,0) circle [radius=0.09];
\draw [fill=black] (2,0) circle [radius=0.09];
\draw [fill=black] (0,2) circle [radius=0.09];
\draw [fill=black] (2,2) circle [radius=0.09];
\node [below] at (0,0) {{$v_{1}$}};
\node [below] at (2,0) {{$v_{2}$}};
\node [right] at (2,2) {{$v_{3}$}};
\node [left] at (0,2) {{$v_{4}$}};
\draw[-] [draw=black, thick] (2.2,0.3)--(5.9,0.3);
\draw[-] [draw=black, thick] (3,0.8)--(6.4,0.8);
\draw[-] [draw=black, thick] (3.6,1.2)--(6.8,1.2);
\draw[-] [draw=black, thick] (4.5,1.6)--(7.8,1.6);
\draw [fill=black] (5.2,0.3) circle [radius=0.05];
\draw [fill=black] (4,0.8) circle [radius=0.05];
\draw [fill=black] (6,1.2) circle [radius=0.05];
\draw [fill=black] (5.5,1.6) circle [radius=0.05];
\node [right] at (5.9,0.3) {{$I_{4}$}};
\node [right] at (6.4,0.8) {{$I_{1}$}};
\node [right] at (6.8,1.1) {{$I_{2}$}};
\node [right] at (7.8,1.6) {{$I_{3}$}};
\node [above] at (5.2,0.3) {{$p_{4}$}};
\node [above] at (4,0.8) {{$p_{1}$}};
\node [above] at (6,1.2) {{$p_{2}$}};
\node [above] at (5.5,1.6) {{$p_{3}$}};
\end{tikzpicture}
\caption{$C_{4}$ with its proper MPTG representation}\label{notinterval}
\end{minipage}
\end{figure}

\noindent In the following Lemma we present a necessary condition for proper MPTG graphs.

\begin{lem}\label{lm1}
Let $G=(V,E)$ be a proper MPTG. Then there exist an ordering of vertices of $V$ which satisfy the following:

\begin{enumerate}

\item For any $v_{1},v_{2},v_{3}\in V$, $v_{1}<v_{2}<v_{3}, v_{1}v_{3}\in E \Rightarrow v_{1}v_{2}\in E$ or $v_{2}v_{3}\in E$.	($3$-point condition)\label{3pt}

\item For any $v_{1},v_{2},v_{3},v_{4},v_{5}\in V$, $v_{1}<v_{2}<v_{3}<v_{4}<v_{5}, v_{1}v_{4},v_{2}v_{5}\in E \Rightarrow v_{1}v_{3}\in E$ or $v_{3}v_{5}\in E$ when $v_{1}v_{2}$ or $v_{4}v_{5}\in E$.  ($5$-point condition-$1$)\label{5pt1}

\item For any $v_{1},v_{2},v_{3},v_{4},v_{5}\in V$, $v_{1}<v_{2}<v_{3}<v_{4}<v_{5}, v_{1}v_{3},v_{3}v_{5}\in E \Rightarrow v_{1}v_{2}\in E$ or $v_{2}v_{4}\in E$ or $v_{4}v_{5}\in E$. ($5$-point condition-$2$)\label{5pt2}

\item For any $v_{1},v_{2},v_{j},v_{k},v_{5},v_{6}\in V$, $v_{1}<v_{2}<v_{j},v_{k}<v_{5}<v_{6}$, $v_{1}v_{j}, v_{j}v_{5},v_{2}v_{k},v_{k}v_{6}\in E$ for some $j,k$ such that $2<j,k<5 \Rightarrow v_{1}v_{2}\in E$ or $v_{2}v_{5}\in E$ or $v_{5}v_{6}\in E$. ($6$-point condition)\label{6pt}
\end{enumerate}
\end{lem}

\begin{proof}
Let $G$ be a proper MPTG with representation 
$\{(I_{i},p_{i})|i\in V\}$ where $I_{i}=[a_{i},b_{i}]$ be an interval on real line and $p_{i}$ is a point within it. We arrange the vertices of $V$ according to increasing order of $p_{i}$'s. Then the following holds.

\noindent In condition $(1)$ as $v_{1}v_{3}\in E$, $a_{3}\leq p_{1}<p_{2}<p_{3}\leq b_{1}$ which imply $p_{2}\in [p_{1},p_{3}]\subseteq I_{1}, I_{3}$. Now if $v_{1}v_{2},v_{2}v_{3}\notin E$, $p_{1},p_{3}\notin I_{2}$ which imply $p_{1}<a_{2}$ and $b_{2}<p_{3}$ as $p_{1}<p_{2}<p_{3}$. Hence $a_{3}\leq p_{1}<a_{2}<b_{2}<p_{3}$ which imply $[a_{2},b_{2}]\subset [a_{3},p_{3}]\subseteq I_{3}$ which contradicts that the intervals are proper. Hence the result follows. 
\vspace{0.1em}

\noindent Next in condition $(2)$ as $v_{1}v_{4}\in E$, $a_{4}\leq p_{1}<p_{2}<p_{3}<p_{4}\leq b_{1}$, which imply $p_{2},p_{3}\in I_{1},I_{4}$. Again as $v_{2}v_{5}\in E$, $a_{5}\leq p_{2}<p_{3}<p_{4}<p_{5}<b_{2}$ which imply $p_{3},p_{4}\in I_{2},I_{5}$. Now if $v_{1}v_{3}, v_{3}v_{5}\notin E$, $p_{1},p_{5}\notin I_{3}$ which imply $p_{1}<a_{3}$ and $b_{3}<p_{5}$ as $p_{1}<p_{3}<p_{5}$. If $v_{1}v_{2}\in E$ then $a_{2}\leq p_{1}<a_{3}<b_{3}<p_{5}<b_{2}$ which imply $[a_{3},b_{3}]\subset[a_{2},b_{2}]$ which is a contradiction as the intervals are proper. Now if $v_{4}v_{5}\in E$ then $a_{4}\leq p_{1}<a_{3}<b_{3}<p_{5}\leq b_{4}$ imply $[a_{3},b_{3}]\subsetneq [a_{4},b_{4}]$ which is again a contradiction as the intervals are proper. Thus the condition holds true.
\vspace{0.1em}

\noindent We note in condition $(3)$, $v_{1}v_{3}\in E$ imply $a_{3}\leq p_{1}<p_{2}<p_{3}\leq b_{1}$. Therefore $p_{2}\in I_{1}\cap I_{3}$. Now if $v_{1}v_{2}\notin E$, then $a_{2}>p_{1}$ as $p_{1}<p_{2}$. Now as $v_{3}v_{5}\in E$, $a_{5}\leq p_{3}<p_{4}<p_{5}\leq b_{3}$. This imply $p_{4}\in I_{3}\cap I_{5}$. Now if $v_{4}v_{5}\notin E$, then $b_{4}<p_{5}$ as $p_{4}<p_{5}$. Now as $p_{2},p_{4}\in I_{3}$ if $v_{2}v_{4}\notin E$ imply either $b_{2}<p_{4}$ or $p_{2}<a_{4}$ as $p_{2}<p_{3}<p_{4}$. If $b_{2}<p_{4}$ then $a_{3}\leq p_{1}<a_{2}<b_{2}<p_{4}<p_{5}\leq b_{3}$ which imply $[a_{2},b_{2}]\subset [a_{3},b_{3}]$ which is a contradiction as the intervals are proper. Again if $p_{2}<a_{4}$ then $a_{3}\leq p_{1}<p_{2}<a_{4}<b_{4}<p_{5}\leq b_{3}$ which imply $[a_{4},b_{4}]\subsetneq [a_{3},b_{3}]$ which is again a contradiction as the intervals are proper. Hence the proof holds. 
\vspace{0.1em}

\noindent Again in condition $(4)$ as $v_{1}v_{j}\in E$, $a_{j}\leq p_{1}<p_{2}<p_{j}\leq b_{1}$, which imply $p_{2}\in I_{1},I_{j}$. Now if $v_{1}v_{2}\notin E$, then $a_{2}>p_{1}$ as $p_{1}<p_{2}$. Again as $v_{k}v_{6}\in E$, $a_{6}\leq p_{k}<p_{5}<p_{6}\leq b_{k}$, which imply $p_{5}\in I_{k}, I_{6}$. Now if $v_{5}v_{6}\notin E$, then $b_{5}<p_{6}$ as $p_{5}<p_{6}$. If $v_{2}v_{5}\notin E$ then either $p_{2}<a_{5}$ or $b_{2}<p_{5}$ as $p_{2}<p_{5}$. Now if $p_{2}<a_{5}$ then $a_{k}\leq p_{2}<a_{5}\leq p_{j}<p_{5}<b_{5}<p_{6}\leq b_{k}$ as $v_{2}v_{k}, v_{j}v_{5}\in E$, which imply $[a_{5},b_{5}]\subsetneq [a_{k},b_{k}]$ which is a contradiction as the intervals are proper. Again if $b_{2}<p_{5}$ then $a_{j}\leq p_{1}<a_{2}<b_{2}<p_{5}\leq b_{j}$ as $v_{j}v_{5}\in E$, which imply $[a_{2},b_{2}]\subsetneq [a_{j},b_{j}]$ which again contradicts that the intervals are proper. Hence the result follows.
\end{proof}

\begin{defn}\label{pmptgordering}
Let $G=(V,E)$ be an undirected graph and the vertices of $G$ satisfy condition \ref{4pt} of Theorem \ref{mptg1} and all the conditions of Lemma \ref{lm1} with respect to a vertex ordering of $V$. Then we call such an ordering as {\em proper MPTG ordering} of $V$. Let $\sigma$ be a proper MPTG ordering and $A^{*}(G)=(a_{i,j})$ be the augmented adjacency matrix of $G$.
From Theorem \ref{mptg1} it is clear that $G$ is an MPTG with respect to $\sigma$. Let $\{(I_{i},p_{i})|i\in V\}$ be a MPTG representation of $G$. Below we define a order relation $\prec$ among the right end points of each interval $I_{i}=[a_{i},b_{i}]$ associated to every vertex (say $v_{i}$) of $G$ when vertices are arranged in $\sigma$ order in $A^{*}(G)$ along both rows and columns of it. 

\noindent Let $R_{1},R_{2}$ be two partial order relations defined on the set $\{b_{i}| i\in V\}$. For $i<j$ in $A^{*}(G)$ we define,
\begin{eqnarray}\label{1}
 b_{j} R_{1} b_{i} \hspace*{0.7 em} \mbox{if there exists some} \hspace*{0.7 em} k_{2}>j \hspace*{0.7 em} \mbox{such that} \hspace*{0.7 em} a_{j,k_{2}}=0 \footnotemark \hspace*{0.7 em} \mbox{and} \hspace*{0.7 em} a_{i,k_{2}}=1 
\end{eqnarray}\footnotetext{ This zero is right open from condition $4$ of Theorem \ref{mptg1}}
 \begin{eqnarray}\label{2}
b_{j}R_{2}b_{i} \hspace*{0.7 em} \mbox{if there exists some} \hspace*{0.7 em} k_{1}<i \hspace*{0.7 em}\mbox{such that} \hspace*{0.7 em} a_{k_{1},i}=0\footnotemark \hspace*{0.7 em} \mbox{and} \hspace*{0.7 em} a_{k_{1},j}=1. 
 \end{eqnarray}\footnotetext{ This zero is up open from condition $4$ of Theorem \ref{mptg1}}
 
\noindent Now we define a order $`\prec$' between $b_{i}$'s in the following way. For $i<j$, 
\begin{eqnarray}\label{3}
b_{j}\prec b_{i}\hspace*{0.5 em} \mbox{if} \hspace*{0.5 em}b_{j} R_{1} b_{i} \hspace*{0.5 em}\mbox{or}\hspace*{0.5 em} b_{j} R_{2} b_{i} \hspace*{0.5 em}\mbox{or}\hspace*{0.5 em} b_{j} R_{1} b_{k} \hspace*{0.5 em}\mbox{and} \hspace*{0.5 em} b_{k} R_{2} b_{i} \hspace*{0.5 em}\mbox{for some} \hspace*{0.5 em}k\in V \hspace*{0.5 em}\mbox{satisfying}\hspace*{0.5 em} i<k<j.
\end{eqnarray}

\begin{eqnarray}\label{4}
b_{i}\prec b_{j} \hspace*{0.7 em}\mbox{otherwise}.
\end{eqnarray}

\noindent It is important to note for $i<j$, $b_{i}\prec b_{j}$ whenever $a_{i,j}=0$. If not let $b_{j}\prec b_{i}$, then from the above definition of $\prec$,
one of the following holds, $b_{j}R_{1} b_{i}$, $b_{j} R_{2} b_{i}$ or 
$b_{j}R_{1} b_{k}$ and $b_{k} R_{2} b_{i}$ for some $k\in V$ such that $i<k<j$. In first two cases condition \ref{3pt} of Lemma \ref{lm1} and in last case condition \ref{5pt2} of Lemma \ref{lm1} gets contradicted.
\end{defn}

\noindent In the following Lemma we will prove $\prec$ to be a total order.  

\begin{lem}\label{lm2}
Let $G=(V,E)$ be a MPTG satisfying conditions Lemma \ref{lm1} with respect to an ordering of vertices of $V$. Then $\prec$ as described in Definition \ref{pmptgordering} is a total order.
\end{lem}

\begin{proof}
Let $G$ satisfies conditions of Lemma \ref{lm1} with respect to a ordering (say $\sigma$) of vertices of $V$ having MPTG representation $\{(I_{v},p_{v})|v\in V\}$ where $I_{v}=[a_{v},b_{v}]$. Let the augmented adjacency matrix of $G$ (say $A^{*}(G)=(a_{i,j}))$ is arranged according to $\sigma$ along both rows and columns. Now to show $\prec$ is a total order it is sufficient to prove $`\prec$' is transitive. Let $b_{i_{1}}\prec b_{i_{2}}$ and $b_{i_{2}}\prec b_{i_{3}}$. Below we will show $b_{i_{1}}\prec b_{i_{3}}$. Several cases may arise depending on occurrence of $i_{1},i_{2},i_{3}$ in $\sigma$. 

\vspace{0.3em}
\noindent \textbf{Case (i) }\  $\bf{i_{1}<i_{3}<i_{2}}$
 
\noindent We assume on contrary $b_{i_{3}}\prec b_{i_{1}}$. We use the following observations repeatedly for the rest of the proof.

\noindent \textit{Observations:}
\begin{enumerate}
\item First we note $a_{i_{3},i_{2}}=1$ as $i_{3}<i_{2}$. 
\item Next $b_{i_{3}} R_{1} b_{j}$ for any $j<i_{3}$ imply existence of some $\bf{k_{2}^{'}}>i_{3}$ such that $a_{i_{3},k_{2}^{'}}=0,a_{j,k_{2}^{'}}=1$. Moreover $k_{2}^{'}>i_{2}$ otherwise (\ref{4p1}) gets contradicted for vertices $\{j,i_{3},k_{2}^{'},i_{2}\}$.
\item Next we see $b_{i_{2}} R_{1} b_{r}$ for any $r<i_{2}$ imply existence of some $\bf{k_{2}}>i_{2}$ such that $a_{i_{2},k_{2}}=0, a_{r,k_{2}}=1$. Note that if $b_{i_{3}} R_{1} b_{j}$ for any $j<i_{3}$ then if $k_{2}^{'}>k_{2}$,  $a_{i_{2},k_{2}^{'}}$ becomes zero as $a_{i_{2},k_{2}}=0$ is right open, which helps us to conclude $b_{i_{2}} R_{1} b_{j}$ from (\ref{1}).
\item  Similarly $b_{l} R_{2} b_{i_{1}}$ for any $l>i_{1}$ ensure existence of some $\bf{k_{1}^{'}}<i_{1}$ such that $a_{k_{1}^{'},i_{1}}=0,a_{k_{1}^{'},l}=1$.
\item Again  $b_{m} R_{2} b_{i_{3}}$ for any $m>i_{3}$ ensure exist of some $\bf{k_{1}}<i_{3}$ satisfying $a_{k_{1},i_{3}}=0,a_{k_{1},m}=1$. We note that $k_{1}<i_{1}$ otherwise as $a_{k_{1},i_{3}}=0$ is up open, $a_{i_{1},i_{3}}$ becomes zero which imply $b_{i_{1}}\prec b_{i_{3}}$ which is a contradiction. Now if $b_{l} R_{2} b_{i_{1}}$ holds for any $l>i_{1}$ then if $k_{1}<k_{1}^{'}<i_{1}$, $a_{k_{1},i_{1}}$ becomes zero as $a_{k_{1}^{'},i_{1}}=0$ is upopen, which imply $b_{m} R_{2} b_{i_{1}}$ from (\ref{2}).
\end{enumerate}

\noindent\textit{Main proof:}
Let $b_{i_{2}}R_{1} b_{i_{3}}$. Now if $b_{i_{3}} R_{1} b_{j}$ for any $j<i_{3}$ then we get $k_{2}^{'}$ must be greater than $k_{2}$ otherwise (\ref{4p1}) gets contradicted for vertices $\{j,i_{3},k_{2}^{'},k_{2}\}$. Hence from observation $3$ we get $b_{i_{2}} R_{1} b_{j}$. 

\vspace{0.1em}
\noindent Now if $b_{i_{3}} R_{1} b_{i_{1}}$ we get $b_{i_{2}}\prec b_{i_{1}}$ from last paragraph. Again if $b_{i_{3}}R_{2}b_{i_{1}}$ then $b_{i_{2}}\prec b_{i_{1}}$ follows from (\ref{3}) as $b_{i_{2}} R_{1}b_{i_{3}}$ and $i_{1}<i_{3}<i_{2}$.
 Next if $b_{i_{3}} R_{1} b_{k}$ and $b_{k} R_{2} b_{i_{1}}$ for some $k\in V$ such that $i_{1}<k<i_{3}$, we get $b_{i_{2}} R_{1} b_{k}$ from last paragraph.  Hence from (\ref{3}) we get $b_{i_{2}}R_{1} b_{i_{1}}$ as $b_{k} R_{2} b_{i_{1}}$ and $i_{1}<k<i_{2}$. Thus all the above cases contradicts our assumption.

\vspace{0.2em}

\noindent Let $b_{i_{2}} R_{2} b_{i_{3}}$. If $b_{i_{3}} R_{1}b_{i_{1}}$ then $k_{2}^{'}>i_{2}$ from observation $2$. Hence $a_{i_{2},k_{2}^{'}}=1$ otherwise $b_{i_{2}} R_{1} b_{i_{1}}$ from (\ref{1}) imply $b_{i_{2}}\prec b_{i_{1}}$ which is not true. Hence condition \ref{5pt1} of Lemma \ref{lm1} gets contradicted for vertices $\{k_{1},i_{1},i_{3},i_{2},k_{2}^{'}\}$.
Again if $b_{i_{3}} R_{2} b_{i_{1}}$ then from observation $5$ if $k_{1}<k_{1}^{'}<i_{1}$ we get $b_{i_{2}}\prec b_{i_{1}}$ which is not true.
Again $k_{1}^{'}\neq k_{1}$ clearly. Hence $k_{1}^{'}<k_{1}$. But again (\ref{4p1}) gets contradicted for vertices $\{k_{1}^{'},k_{1},i_{3},i_{2}\}$. Now if $b_{i_{3}} R_{1} b_{k}$ and $b_{k} R_{2} b_{i_{1}}$ for some $k\in V$ such that $i_{1}<k<i_{3}$. Then $a_{i_{2},k_{2}^{'}}=1$ otherwise we get $b_{i_{2}} R_{1} b_{k}$ which imply $b_{i_{2}}\prec b_{i_{1}}$ as $b_{k} R_{2} b_{i_{1}}$ and $i_{1}<k<i_{2}$. But we get a contradiction applying condition \ref{5pt1} of Lemma \ref{lm1} on vertices $\{k_{1},k,i_{3},i_{2},k_{2}^{'}\}$. 

\vspace{0.2em}

\noindent Let $b_{i_{2}} R_{1} b_{k}$ and $b_{k} R_{2} b_{i_{3}}$ for some $k\in V$ such that $i_{3}<k<i_{2}$.  Now if $b_{i_{3}} R_{1} b_{j}$  for any $j<i_{3}$ then if $k_{2}^{'}>k_{2}$, from observation $3$ we get $b_{i_{2}} R_{1} b_{j}$. 

\vspace{0.1em}
\noindent Now if $b_{i_{3}} R_{1} b_{i_{1}}$ then if $k_{2}^{'}>k_{2}$ from last paragraph we get $b_{i_{2}}\prec b_{i_{1}}$ which is not true. Hence $i_{2}<k_{2}^{'}<k_{2}$. Again $a_{k,k_{2}^{'}}=1$ from (\ref{4p1}) applying on vertices $\{i_{1},k,k_{2}^{'},k_{2}\}$. Hence condition \ref{5pt1} of Lemma \ref{lm1} gets contradicted for vertices $\{k_{1},i_{1},i_{3},k,k_{2}^{'}\}$. Next if $b_{i_{3}} R_{2} b_{i_{1}}$ then if $k_{1}<k_{1}^{'}<i_{1}$ then from observation $5$ we get $b_{k} R_{2} b_{i_{1}}$ and hence from (\ref{3}) we find $b_{i_{2}}\prec b_{i_{1}}$ as $b_{i_{2}} R_{1} b_{k}$ and $i_{1}<k<i_{2}$. Again when $k_{1}^{'}<k_{1}$, (\ref{4p1}) gets contradicted for vertices $\{k_{1}^{'},k_{1},i_{3},k\}$. Now let $b_{i_{3}} R_{1} b_{k^{'}}$ and $b_{k^{'}} R_{2} b_{i_{1}}$  for some $k^{'}\in V$ satisfying $i_{1}<k^{'}<i_{3}$. As $b_{i_{3}} R_{1} b_{k^{'}}$ then if $k_{2}^{'}>k_{2}$ then from last paragraph we get $b_{i_{2}} R_{1} b_{k^{'}}$ which helps us to conclude from (\ref{3}) $b_{i_{2}}\prec b_{i_{1}}$ as $b_{k^{'}} R_{2} b_{i_{1}}$.
Hence $i_{2}<k_{2}^{'}<k_{2}$. Thus in this case condition \ref{5pt1} of Lemma (\ref{lm1}) gets contradicted for vertices $\{k_{1},k^{'},i_{3},k,k_{2}^{'}\}$ (note $a_{k,k_{2}^{'}}=1$ otherwise (\ref{4p1}) gets contradicted for vertices $\{k^{'},k,k_{2}^{'},k_{2}\}$).

\vspace{0.3em}
\noindent The verification of other cases are also long and rigorous. Therefore we put them in Appendix.
\end{proof}

\begin{defn}\label{canseq}
Let $G=(V,E)$ be a MPTG having representation $\{([a_{v},b_{v}],p_{v})|v\in V\}$ satisfying a proper MPTG ordering $\sigma$ of $V$ as described in 
Definition \ref{pmptgordering}. Then from Lemma \ref{lm2} it is clear that observing the adjacencies of vertices in $A^{*}(G)$ one can find a total order $\prec$ between all $b_{v}$'s in such a way so that they can be put together in a single sequence $P_{1}$ (say). Next we arrange $a_{i}$'s in a single sequence $P_{2}$ (say) in the same order. Let $|V|=n$. Then

\vspace{0.3em}

\hspace{4.3em} $P_{1}: b_{\alpha_{1}}\prec \hspace*{.2em} b_{\alpha_{2}}\prec\hdots\prec b_{\alpha_{n}}$ and $P_{2}: a_{\alpha_{1}}\prec \hspace*{.2em} a_{\alpha_{2}}\prec\hdots\prec a_{\alpha_{n}}$ where $1\leq \alpha_{i}\leq n$.

\vspace{0.3em}

\noindent We now combine $a_{i},b_{i},p_{i}$'s in a single sequence $P$ by following rule. Henceforth we call $P$ by {\em canonical sequence} corresponding to $\sigma$.

\begin{enumerate}
\item  First we place all $p_{i}$'s on real line according to the occurrence of $i$'s in $\sigma$. 

\item  Now using induction starting from the last element of $P_{2}$ until the first element is reached we place $a_{i}$'s on real line in following way.

\vspace{0.2em} 
We first place $a_{\alpha_{n}}$ between $p_{(\alpha_{n})_{1}-1}$ and $p_{(\alpha_{n})_{1}}$. 
Now for $a_{\alpha_{k}}$ if $(\alpha_{k})_{1}\geq (\alpha_{k+1})_{1}$ then we place $a_{\alpha_{k}}$ just before $a_{\alpha_{k+1}}$, place $a_{\alpha_{k}}$ between $p_{{(\alpha_{k})}_{1}-1}$ and $p_{{(\alpha_{k})}_{1}}$ otherwise,
where $(\alpha_{i})_{1}$ is the first column containing one in $\alpha_{i}$ th row in $A^{*}(G)$.

\item Again using induction starting from first element of $P_{1}$ until the last element is reached we place $b_{i}$'s on real line by following rule.

First we place $b_{\alpha_{1}}$ between $p_{(\alpha_{1})_{2}}$ and $p_{(\alpha_{1})_{2}+1}$. Now for $b_{\alpha_{k}}$ if 
 $(\alpha_{k})_{2}\leq (\alpha_{k-1})_{2}$ then we place $b_{\alpha_{k}}$ just after $b_{\alpha_{k-1}}$, place $b_{\alpha_{k}}$ between $p_{{(\alpha_{k})}_{2}}$ and $p_{{(\alpha_{k})}_{2}+1}$ otherwise
where $(\alpha_{i})_{2}$ is the last column containing one in $\alpha_{i}$ th row in $A^{*}(G)$. 
\end{enumerate}

\noindent From construction of $P$ one can verify that if there are more than one $b_{i}(a_{i})$'s occur between two $p_{i}$'s then they are arranged according to their occurrence in $P_{1}(P_{2})$ respectively.
Again it is easy to check that the order of the sequences $P_{1},P_{2}$ get preserved in $P$. 
\end{defn}

\begin{exmp}\label{ex1}
Consider the MPTG $G=(V,E)$ in Figure \ref{pm1} whose augmented adjacency matrix $A^{*}(G)$ is arranged according to a proper MPTG ordering $\sigma=(v_{1},\hdots,v_{7})$ of $V$.
Let $[a_{i},b_{i}]$, $p_{i}$ be the interval and a point within it corresponding to vertex $v_{i}$ for $i=1,2,\hdots,7$ in its MPTG representation. Now looking at the adjacencies of vertices in $A^{*}(G)$ we construct $P_{1},P_{2}$ sequences of endpoints $b_{i},a_{i}$ respectively according to the construction described in Definition \ref{canseq}. 

\hspace{12em} \noindent $P_{1}: b_{2}\prec b_{1}\prec b_{5}\prec b_{6}\prec b_{4}\prec b_{3}\prec b_{7}$.

\hspace{12.3em}\noindent $P_{2}: a_{2}\prec a_{1}\prec a_{5}\prec a_{6}\prec a_{4}\prec a_{3}\prec a_{7}$.

\vspace{0.3em}
\noindent The combined sequence of $a_{i},p_{i},b_{i}$ is 
$P: a_{2}\hspace*{.3 em} a_{1} \hspace*{.3 em} a_{5} \hspace*{.3 em} a_{6} \hspace*{.3 em} a_{4} \hspace*{.3 em} p_{1} \hspace*{.3 em} a_{3} \hspace*{.3 em} p_{2} \hspace*{.3 em} p_{3} \hspace*{.3 em} b_{2} \hspace*{.3 em} a_{7} \hspace*{.3 em} p_{4} \hspace*{.3 em} b_{1} \hspace*{.3 em} p_{5} \hspace*{.3 em} b_{5} \hspace*{.3 em} p_{6} \hspace*{.3 em} b_{6} \hspace*{.3 em} p_{7} \hspace*{.3 em} b_{4} \hspace*{.3 em} b_{3} \hspace*{.3 em} b_{7}$.
\end{exmp}
\begin{figure}[t]\centering
$\begin{array}{ccccccccc}

&&6&8&9&12&14&16&18\\
&&v_{1}&v_{2}& v_{3} &v_{4} &v_{5} &v_{6} &v_{7}\\
	\cline{3-9}
	[1,13] &\multicolumn{1}{c@{\,\vline}}{v_{1}} & 1 & 1 & 0 &1 &0 & 0 & \multicolumn{1}{c@{\,\vline}}{0}\\
	
[2,10]&	\multicolumn{1}{c@{\,\vline}}{v_{2}} & 1 & 1 & 1 & 0 & 0 & 0 & \multicolumn{1}{c@{\,\vline}}{0}\\
	
[7,20] &\multicolumn{1}{c@{\,\vline}}{v_{3}} & 0 & 1 & 1 & 1 & 1 & 1 & \multicolumn{1}{c@{\,\vline}}{0} \\

[5,19] &	\multicolumn{1}{c@{\,\vline}}{v_{4}} & 1 & 0 & 1 & 1 & 1 & 1 &  \multicolumn{1}{c@{\,\vline}}{1}\\

[3,15]& \multicolumn{1}{c@{\,\vline}}{v_{5}} & 0 & 0 & 1 & 1 & 1 & 0 &  \multicolumn{1}{c@{\,\vline}}{0}\\

[4,17] &\multicolumn{1}{c@{\,\vline}}{v_{6}} & 0 & 0 & 1 & 1 & 0 & 1 &  \multicolumn{1}{c@{\,\vline}}{0}\\

[11,21] &\multicolumn{1}{c@{\,\vline}}{v_{7}} & 0 & 0 & 0 & 1 & 0 & 0 &  \multicolumn{1}{c@{\,\vline}}{1}\\	
\cline{3-9}
\end{array}
$
\caption{Augmented adjacency matrix $A^{*}(G)$ of a proper MPTG $G$}\label{pm1} 
\end{figure}

\noindent We now present the main theorem of our article which characterizes a proper max-point-tolerance graph.

\subsection{\textbf{\textit{Characterization of proper MPTG}}}

\begin{thm}\label{pmptg1}
Let $G$ be an undirected graph with vertex set $V$ and edge set $E$. Then $G$ is a proper MPTG if and only if it satisfies a proper MPTG ordering of $V$ as described in Definition \ref{pmptgordering}. 
\end{thm}

\begin{proof}
Let $G=(V,E)$ be a proper MPTG. Then it satisfies (\ref{4p1}) and all other conditions required for a vertex ordering to be a proper MPTG ordering follows from Theorem \ref{mptg1} and Lemma \ref{lm1}.

\vspace{0.3em}
\noindent {\em Conversely}
Let $G$ possesses a proper MPTG ordering (say $\sigma$) of $V$. Then it satisfy (\ref{4p1}) and conditions of Lemma \ref{lm1} with respect to $\sigma$. Let the augmented adjacency matrix of $G$, i.e; $A^{*}(G)=(a_{ij})$ is arranged according to $\sigma$. Then from Theorem \ref{mptg1} $G$ becomes a MPTG with respect to $\sigma$.
Let $[a_{i},b_{i}]$ be the interval and $p_{i}$ be a point within it associated to each vertex $i\in V$ in its MPTG representation.
Now from the adjacencies of the vertices in $A^{*}(G)$ we define a order relation $\prec$ among $b_{i}$'s as described in Definition \ref{pmptgordering}. Moreover in Lemma \ref{lm2} we prove $\prec$ as a total order. Next we construct sequences $P_{1}(P_{2})$ consists of $b_{i}(a_{i})$ as described in Definition \ref{canseq}. Then we construct the canonical sequence $P$ corresponding to $\sigma$, which is basically a combined sequence containing all $a_{i},b_{i},p_{i}$'s. Now we associate the numbers $1$ to $3n$ ($|V|=n$) to $P$ starting from first element of it and define $I_{i}=[\bar{a_{i}},\bar{b_{i}}]$ and $p_{i}=\bar{p_{i}}$ where 
$\bar{a_{i}},\bar{b_{i}},\bar{p_{i}}$ denote the numbers associated to $a_{i},b_{i},p_{i}$ in $P$. Below we prove $\{(I_{i},p_{i})|i\in V\}$ will actually give proper MPTG representation of $G$.
 
 \vspace{0.3em}
\noindent \textbf{Step (i)} First we verify that the intervals $I_{i}$ are proper. 

\noindent First we show $a_{i}<b_{i}$ for each $i$ in $P$. From construction of $P$ one can easily check that each $a_{i}$ is placed left to $p_{i_{1}}$ and each $b_{i}$ is placed right to $p_{i_{2}}$ where $(i)_{1},(i)_{2}$ denote columns containing first and last one in $i$ th row of $A^{*}(G)$. Now as $p_{(i)_{1}}< p_{i}< p_{(i)_{2}}$ in $P$, $a_{i}<b_{i}$.

\noindent Now as $b_{i}, a_{i}$'s are arranged in $P_{1},P_{2}$ in same order and they individually keep their orderings intact in $P$, no interval can be contained in other. Hence the intervals $\{I_{i}=[\bar{a_{i}},\bar{b_{i}}]|i\in V\}$ are proper with respect to our given representation.

\noindent \textbf{Step (ii)} Note that $a_{i}<p_{i}<b_{i}$ in $P$ as $(i)_{1}\leq i\leq (i)_{2}$ from Definition \ref{canseq}.

\vspace{0.3em}

\noindent \textbf{Step (iii)} Next we verify that $([\bar{a_{i}},\bar{b_{i}}],\bar{p_{i}})$ is the required proper MPTG representation of $G$. 
\vspace{0.1cm}

\noindent For this it is sufficient to show that 
with respect to the above representation $A^{*}(G)$ satisfies all its adjacency relations.

\vspace{0.1cm} 

\noindent {\textit{Case (i)}}  We show that $a_{i,j}=1$ imply $p_{i},p_{j}\in I_{i}\cap I_{j}$.

\noindent Let for $i<j$, $a_{i,j}=1$. If $b_{i}\prec b_{j}$ in $P_{1}$ (i.e; $a_{i}\prec a_{j}$ in $P_{2}$) then as $(j)_{1}\leq i$, $a_{j}<p_{i}$ in $P$. Again as $(i)_{2}>j$, $p_{j}<b_{i}$ in $P$.
Again if $b_{j}\prec b_{i}$ in $P_{1}$ (i.e; $a_{j}\prec a_{i}$ in $P_{2}$) then as $(j)_{2}\geq j$, $b_{j}>p_{j}$ in $P$. Again $(i)_{1}\leq i$ imply $a_{i}<p_{i}$. 
Now as the order of the elements of $P_{1},P_{2}$ remain intact in $P$ we get $p_{i},p_{j}\in I_{i}\cap I_{j}$. 
\vspace{0.1cm} 

\noindent {\textit{Case (ii)}} Next we show that $a_{i,j}=0$ imply either $p_{i}\notin I_{j}$ or $p_{j}\notin I_{i}$. 
\vspace{0.1em}

\noindent The proof will follow from the following two claims.
\vspace{0.3em}

\noindent \textit{\textbf{Claim $1$}} \textbf{For $i<j$ when $a_{i,j}=0$ is \it{\textbf{only up open}}\footnote{up open but not right open}, $p_{i}<a_{j}$ in $P$}.

\noindent Let for $i<j$, $a_{i,j}=0$ is only up open. Then there exist some $k>j$ such that $a_{i,k}=1$. From condition \ref{3pt} of Lemma \ref{lm1} $a_{j,k}=1$. Again $b_{i}\prec b_{j}$ as $i<j$ in $P_{1}$. Now as $(i)_{2}\geq k>j$, $b_{i}>p_{k}>p_{j}$. Again we note $(j)_{1}>i$. Hence to show $a_{j}>p_{i}$ in $P$ it is sufficient to prove that if there exist some $l\in V$ such that $a_{j}\prec a_{l}$ in $P_{2}$ then $(l)_{1}>i$.
\vspace{0.1em}

\noindent On contrary lets assume $a_{j}\prec a_{l}$ for some $l$ such that $(l)_{1}\leq i$. Several cases may arise depending upon the occurrence of $l$ in $A^{*}(G)$. Note that $l$ can not be equal to $i$ as $a_{l}\prec a_{j}$ in $P_{2}$ then contradicts our assumption.
Again $l\neq k$ otherwise $b_{l} R_{2} b_{j}$ from (\ref{2})
imply $b_{l}\prec b_{j}$ in $P_{1}$ and hence we get $a_{l}\prec a_{j}$ in $P_{2}$ which introduces contradiction.
 Next if $l<i$ then $a_{l,j}$ becomes zero as $a_{i,j}=0$ is up open, which imply $b_{l}\prec b_{j}$ in $P_{1}$ and thus we get $a_{l}\prec a_{j}$ in $P_{2}$ which again contradicts our assumption.

\vspace{0.1em}

\noindent Now we consider the case when $i<l<j$. We need the to use the following observations to find contradiction in this case. 

\noindent\textit{Observations:} 
\begin{enumerate}
\item As $(l)_{1}\leq i$, $a_{i,l}$ becomes one from (\ref{4p1}) applying on vertices $\{(l)_{1},i,l,k\}$. 

\item If $b_{j} R_{1} b_{r}$ for any $r<j$ then there exist some $\bf{k_{2}}>j$ such that $a_{j,k_{2}}=0, a_{r,k_{2}}=1$.  We note that $k_{2}>k$ as otherwise (\ref{4p1}) get contradicted for vertices $\{r,j,k_{2},k\}$. 

\item Again if $b_{m} R_{2} b_{l}$ for any $m>l$ then there exist some $\bf{k_{1}}<l$ such that $a_{k_{1},l}=0, a_{k_{1},m}=1$. If $(l)_{1}<k_{1}<l$ then $\{(l)_{1},k_{1},l,m\}$ contradict (\ref{4p1}). Hence $k_{1}<(l)_{1}\leq i$.
\end{enumerate}

\noindent \textit{Main proof:} Now if $b_{j} R_{1} b_{l}$ then from observation $2$ we get 
$k_{2}>k$. One can verify now that condition \ref{5pt1} of Lemma \ref{lm1} get contradicted for vertices $\{i,l,j,k,k_{2}\}$ using observation $1$. Again if $b_{j} R_{2} b_{l}$ then $k_{1}<(l)_{1}\leq i$ from observation $3$. But in this case also (\ref{4p1}) gets contradicted for vertices $\{k_{1},i,j,k\}$.
Now if $b_{j} R_{1} b_{k^{'}}$ and $b_{k^{'}} R_{2} b_{l}$ for some $k^{'}\in V$ such that $l<k^{'}<j$ then again $k_{2}$ must be greater than $k$ and $k_{1}<(l)_{1}\leq i$ from observations $2,3$. This imply $a_{i,k^{'}}=1$ applying 
 $(\ref{4p1})$ on vertex set $\{k_{1},i,k^{'},k\}$
But in this case condition \ref{5pt1} of Lemma \ref{lm1} gets contradicted for vertices $\{i,k^{'},j,k,k_{2}\}$.

\vspace{0.2em}

\noindent Now if $l>j$ then if $(l)_{1}<i$, $a_{(l)_{1},j}=0$ as $a_{i,j}=0$ is up open. This imply $b_{l} R_{2} b_{j}$ from (\ref{2}) i.e; $b_{l}\prec b_{j}$. Thus we get $a_{l}\prec a_{j}$, which is a contradiction. Again for $(l)_{1}=i$ we get similar contradiction. 

\vspace{0.1cm} 
\noindent Hence for any $l$, $a_{j}\prec a_{l}$ in $P_{2}$ imply $(l)_{1}>i$. Thus from our above claim $p_{i}<a_{j}$ in $P$ is established and therefore $p_{i}\notin I_{j}$.

\vspace{0.3em}
\noindent \textit{\textbf{Claim $2$}} \textbf{For $i<j$ when $a_{i,j}=0$ is right open then either $p_{i}<a_{j}$ or $b_{i}<p_{j}$ in $P$.}

\noindent Let for $i<j$, $a_{i,j}=0$ is right open. Then $b_{i}\prec b_{j}$ in $P_{1}$. As $(i)_{2}<j$ it is sufficient to prove that either for all $r\in V$ satisfying $b_{r}\prec b_{i}$ in $P_{1}$, $(r)_{2}<j$ which imply $b_{i}<p_{j}$ or for all $l\in V$ such that $a_{j}\prec a_{l}$ in $P_{2}$, $(l)_{1}>i$ which imply $p_{i}<a_{j}$. 

\vspace{0.1em}

\noindent We assume on contrary existence of some $r$ satisfying $b_{r}\prec b_{i}$ and $(r)_{2}\geq j$. Note that $r$ can not be greater than $j$ as in that case $a_{i,r}$ becomes zero due to the fact $a_{i,j}=0$ is right open, which imply $b_{i}\prec b_{r}$ which contradicts our assumption.
Again when $r<i$, as $(r)_{2}\geq j$ and $a_{i,j}=0$ is right open, $a_{i,(r)_{2}}$ becomes zero, which imply $b_{i} R_{1} b_{r}$ from (\ref{1}) and hence we get $b_{i}\prec b_{r}$ which again contradicts our assumption. Similar contradiction will occur when $r=j$. Hence $i<r<j$.

\vspace{0.2 em}

\noindent \textit{Case (i)} Now if $b_{r} R_{1} b_{i}$, then there exist $k_{2}>r$ such that $a_{r,k_{2}}=0, a_{i,k_{2}}=1$. Note that $k_{2}$ can not occur right to $j$ in $A^{*}(G)$ as $a_{i,j}=0$ is right open. Hence $r<k_{2}<j$.
Now one can find $a_{r,j}=0$ as $a_{r,k_{2}}=0$ is right open. Hence $(r)_{2}>j$, which again contradict  (\ref{4p1}) for vertices $\{i,r,k_{2},(r)_{2}\}$.

\noindent Thus for any $r\in V$ if $b_{r} R_{1} b_{i}$ holds then we get $(r)_{2}<j$. Hence one can conclude $b_{i}<p_{j}$ as $(i)_{2}<j$ in this case.
 
\vspace{0.3 em}

\noindent \textit{Case (ii)} Again if $b_{r} R_{2} b_{i}$ then there exist some $k_{1}<i$ such that $a_{k_{1},i}=0, a_{k_{1},r}=1$. From condition \ref{3pt} of Lemma \ref{lm1}, $a_{i,r}=1$. We will show in this case $p_{i}<a_{j}$. For this we assume on contrary $a_{j}\prec a_{l}$ for some $l\in V$ such that $(l)_{1}\leq i$. Using Lemma \ref{lm2} we get
\begin{equation}\label{e1}
b_{r}\prec b_{i}\prec b_{j}\prec b_{l}
\end{equation}

\noindent In the following paragraph we list some important observations  which we vastly use in rest of the proof. 

\vspace{0.3 em}
\noindent \textit{Observations:} 
\begin{enumerate}
\item For $l<j$, $a_{l,j}=1$, for $l<r$, $a_{l,r}=1$ and for $i<r$, $a_{i,r}=1$ otherwise (\ref{e1}) gets contradicted.

\item For $i<l$, $k_{1}<(l)_{1}\leq i$. If $(l)_{1}<k_{1}$ then $a_{(l)_{1},i}$ becomes zero as $a_{k_{1},i}=0$ is up open which imply $b_{l} R_{2} b_{i}$ from (\ref{2}), i.e; $b_{l}\prec b_{i}$ which is not true from (\ref{e1}). Hence $k_{1}<(l)_{1}\leq i$. 

\item Next we show $a_{r,j}=1$. We assume on contrary that $a_{r,j}=0$ then $(r)_{2}>j$.

\noindent When $l<r$ following observation $1$ one can see that (\ref{4p1}) gets contradicted for vertices $\{l,r,j,(r)_{2}\}$.
Next when $l>j$, $a_{(l)_{1},j}=1$ otherwise $b_{l} R_{2} b_{j}$ (i.e; $b_{l}\prec b_{j}$) follows from (\ref{2}) and the fact $k_{1}<(l)_{1}\leq i$ (observation $2$) which contradicts (\ref{e1}). Hence applying (\ref{4p1}) on vertices $\{(l)_{1},r,j,(r)_{2}\}$ one can get contradiction.

\noindent Again when $r<l<j$ the following cases lead us to contradiction.

\noindent If $b_{j} R_{1} b_{l}$ holds then $a_{j,k_{2}^{'}}=0, a_{l,k_{2}^{'}}=1$ for some $k_{2}^{'}>j$. Now if $(r)_{2}>k_{2}^{'}$ then $a_{j,(r)_{2}}$ becomes zero as $a_{j,k_{2}^{'}}$ is right open and hence from (\ref{1}) we get $b_{j} R_{1} b_{r}$, i.e; $b_{j}\prec b_{r}$ which contradicts (\ref{e1}). Again if $j<(r)_{2}<k_{2}^{'}$, then $a_{r,l}=1$ follows from (\ref{4p1}) applying on vertices $\{(l)_{1},r,l,(r)_{2}\}$
(observation $2$). Hence from observation $6$ one can see that condition \ref{5pt1} of Lemma \ref{lm1} gets contradicted for $\{r,l,j,(r)_{2},k_{2}^{'}\}$.

\noindent Again if $b_{j} R_{2} b_{l}$ then there exist $k_{1}^{'}<l$ such that $a_{k_{1}^{'},l}=0, a_{k_{1}^{'},j}=1$. Hence (\ref{4p1}) gets contradicted for $\{k_{1}^{'},r,j,(r)_{2}\}$.

\noindent Now if $b_{j} R_{1} b_{k^{'}}$ and $b_{k^{'}} R_{2} b_{l}$ for some $k^{'}\in V$ such that $l<k^{'}<j$. Then there exist some $k_{2}^{'}>j,k_{1}^{'}<l$ satisfying $a_{j,k_{2}^{'}}=1, a_{k^{'},k_{2}^{'}}=1,a_{k_{1}^{'},l}=0, a_{k_{1}^{'},k^{'}}=1$. Similar contradiction will arise when $(r)_{2}>k_{2}^{'}$ as described in last paragraph. Now when $j<(r)_{2}<k_{2}^{'}$, $a_{r,k^{'}}=1$ applying (\ref{4p1}) on vertices $\{k_{1}^{'},r,k^{'},(r)_{2}\}$. Hence applying condition $2$ of Lemma \ref{lm1} on vertex set $\{r,k^{'},j,(r)_{2},k_{2}^{'}\}$ one can find contradiction.

\item For $i<l<j$, $a_{i,l}=1$. When $i<l<r$, using observations $1,2$ and applying (\ref{4p1}) on vertices $\{(l)_{1},i,l,r\}$ we get $a_{i,l}=1$. Next if $i<r<l$, $a_{r,l}=1$ (see proof of observation $3$). Now if we assume on contrary $a_{i,l}=0$ then $(l)_{1}<i$. But then condition \ref{5pt1} of Lemma \ref{lm1} gets contradicted for vertices $\{k_{1},(l)_{1},i,r,l\}$.

\item Again for $i<l<j$ if $b_{j} R_{1} b_{m}$ and $b_{m} R_{2} b_{l}$ holds for some $m$ such that $l<m<j$ then $k_{1}<k_{1}^{'}<i$ and $a_{i,m}=1$ where
$a_{k_{1}^{'},l}=0, a_{k_{1}^{'},m}=1$ for some $k_{1}^{'}<l$.  

\noindent If $k_{1}^{'}<k_{1}$ then $a_{k_{1}^{'},i}$ becomes zero as $a_{k_{1},i}=0$ is up open which imply $b_{m} R_{2} b_{i}$ and hence we get $b_{m}\prec b_{j}$ from Lemma \ref{lm2} (as $b_{i}\prec b_{j}$) which is not true from above. Again $k_{1}^{'}\neq i$ as $a_{i,l}=1$ from observation $4$. Now when $i<k_{1}^{'}<l$, (\ref{4p1}) gets contradicted for vertices $\{i,k_{1}^{'},l,m\}$.  Hence $k_{1}<k_{1}^{'}<i$.

\noindent Now if $m<r$ using observation $1$ we get $a_{i,m}=1$   applying (\ref{4p1}) on vertices $\{k_{1}^{'},i,m,r\}$. Again when $m>r$
from observation $3$ and (\ref{4p1}) applying on vertices $\{k_{1}^{'},r,m,j\}$ one can get $a_{r,m}=1$. Now applying condition \ref{5pt1} of Lemma \ref{lm1} on vertices $\{k_{1},k_{1}^{'},i,r,m\}$ we get $a_{i,m}=1$.
\end{enumerate}

\vspace{0.3em}

\noindent \textit{Main proof:} If $l<{k_{1}}$ then $a_{l,i}=0$ as $a_{k_{1},i}=0$ is up open, which imply $b_{l}\prec b_{i}$ as $l<i$. But this is not true from (\ref{e1}).
 Next if $k_{1}<l<i$. Now using observations $1,3$ one can verify that condition \ref{5pt1} of Lemma \ref{lm1} gets contradicted for vertices $\{k_{1},l,i,r,j\}$. 
Again when $j<l$, as $a_{i,j}=0$ is right open $a_{i,l}$ becomes zero. Hence from observation $2$ we get $k_{1}<(l)_{1}<i$. One can show now that condition \ref{5pt1} of Lemma \ref{lm1} gets contradicted for vertices $\{k_{1},(l)_{1},i,r,l\}$ using the fact $a_{r,l}=1$ (see proof of observation $3$). 

\vspace{0.2em}

\noindent  Next we consider the case $\bf{i<l<j}$. Now if $b_{j} R_{1} b_{l}$ then there exist $k_{2}^{'}>j$ such that $a_{j,k_{2}^{'}}=0, a_{l,k_{2}^{'}}=1$. From observations $3,4$ we get $a_{r,j}=1$ and $a_{i,l}=1$. Now applying condition \ref{6pt} of Lemma \ref{lm1} on vertex set $\{k_{1},i,l,r,j,k_{2}^{'}\}$ one can find contradiction. 

\vspace{0.2em}

\noindent  Next if $b_{j} R_{2} b_{l}$ then there exist $k_{1}^{'}<l$ such that $a_{k_{1}^{'},l}=0, a_{k_{1}^{'},j}=1$. Proceeding similarly as observation $5$ and replacing $m$ by $l$ one can prove $k_{1}<k_{1}^{'}<i$. Again from observation $3$ we get $a_{r,j}=1$. Hence one can find contradiction from condition \ref{5pt1} of Lemma \ref{lm1} applying on vertices $\{k_{1},k_{1}^{'},i,r,j\}$.

\vspace{0.2em}

\noindent Again if $b_{j} R_{1} b_{k^{'}}$ and $b_{k^{'}} R_{2} b_{l}$ for some $k^{'}\in V$ such that $l<k^{'}<j$.
 Then there exist $k_{2}^{'}>j, k_{1}^{'}<l$ such that $a_{j,k_{2}^{'}}=0, a_{k^{'},k_{2}^{'}}=1, a_{k_{1}^{'},l}=0, a_{k_{1}^{'},k^{'}}=1$. 
From observation $5$ one can find $k_{1}<k_{1}^{'}<i$. 
Again note $a_{r,j}=1$ from observation $3$ and $a_{i,k^{'}}=1$ from observation $5$. Hence applying condition \ref{6pt} of Lemma \ref{lm1} on vertices $\{k_{1},i,k^{'},r,j,k_{2}^{'}\}$ one can find contradiction.

\vspace{0.1em}
\noindent Thus in this case we prove if there exist some $r\in V$ satisfying $b_{r}R_{2} b_{i}$ such that $(r)_{2}\geq j$ then for any $l\in V$, $a_{j}\prec a_{l}$ in $P_{1}$ imply $(l)_{1}>i$. Hence $p_{i}<a_{j}$ in $P$ and therefore $p_{i}\notin I_{j}$.

\vspace{0.3em} 

\noindent \textit{Case (iii)}
If $b_{r} R_{1} b_{k^{'}}$ and $b_{k^{'}} R_{2} b_{i}$ for some $k^{'}\in V$ such that $i<k^{'}<r$. Then there exist $k_{2}>r, k_{1}<i$ such that $a_{r,k_{2}}=0, a_{k^{'},k_{2}}=1, a_{k_{1},i}=0, a_{k_{1},k^{'}}=1$. Note that $k_{2}>j$ otherwise (\ref{4p1}) gets contradicted for vertex set $\{k^{'},r,k_{2},(r)_{2}\}$ as $(r)_{2}\geq j$. Hence $(k^{'})_{2}\geq k_{2}>j$. Now as $b_{k^{'}} R_{2} b_{i}$, $i<k^{'}<j$ and $(k^{'})_{2}>j$,  proceeding similarly as Case (ii) just replacing $r$ by $k^{'}$ one can find same result.
\end{proof}

\begin{rmk}
In Example \ref{ex1} we found
the canonical sequence $P$ of a graph $G$ from its augmented adjacency matrix $A^{*}(G)$, arranged according to a proper MPTG ordering as described in Figure \ref{pm1}.
 
$\hspace*{9em} P: a_{2}\hspace*{.3 em} a_{1} \hspace*{.3 em} a_{5} \hspace*{.3 em} a_{6} \hspace*{.3 em} a_{4} \hspace*{.3 em} p_{1} \hspace*{.3 em} a_{3} \hspace*{.3 em} p_{2} \hspace*{.3 em} p_{3} \hspace*{.3 em} b_{2} \hspace*{.3 em} a_{7} \hspace*{.3 em} p_{4} \hspace*{.3 em} b_{1} \hspace*{.3 em} p_{5} \hspace*{.3 em} b_{5} \hspace*{.3 em} p_{6} \hspace*{.3 em} b_{6} \hspace*{.3 em} p_{7} \hspace*{.3 em} b_{4} \hspace*{.3 em} b_{3} \hspace*{.3 em} b_{7}$.

\vspace{.3em}
\noindent Therefore we get the proper MPTG representation $([a_{i},b_{i}],p_{i})$ of $G$ from the following table according to Theorem \ref{pmptg1}.
$$\begin{array}{|c c c c c c c c c c c c c c c c c c c c c|}
\hline

a_{2}& a_{1} & a_{5} & a_{6} & a_{4} & p_{1} & a_{3} & p_{2} & p_{3} & b_{2} & a_{7} & p_{4} & b_{1} & p_{5} & b_{5} & p_{6} & b_{6} & p_{7} & b_{4} & b_{3} & b_{7}\\

 1& 2 & 3 & 4 & 5 & 6 & 7 & 8 & 9 & 10 & 11 & 12 & 13 & 14 & 15 & 16 & 17 & 18 & 19 & 20 & 21\\
 \hline
\end{array}
$$
\end{rmk}

\noindent Catanzaro et. al \cite{Catanzaro} have shown that the interval graphs form a strict subclass of MPTG. We extend this result to the class of proper MPTG graphs.

\begin{cor}
Interval graphs form a strict subclass of proper MPTG.
\end{cor}

\begin{proof}
Let $G=(V,E)$ be an interval graph having representation $\{I_{v}=[a_{v},b_{v}]|v\in V\}$. Now arranging the intervals according to the increasing order of left end points ($a_{v}$) and ordering vertices of $V$ with the same it is easy to check that $G$ satisfies (\ref{4p1}) and conditions of Lemma \ref{lm1}. Thus from  Theorem \ref{pmptg1} and the fact that $C_{4}$ is a proper MPTG (see Figure \ref{notinterval}) but not an interval graph one can conclude the result. 
\end{proof}

\noindent In \cite{Bogart1} Bogart and West  proved that proper interval graphs are same as unit interval graphs. Here we can also conclude a similar result for the class of proper MPTG graphs.

\begin{prop}
Let $G=(V,E)$ be an undirected graph. Then the following are equivalent:
\begin{enumerate}
\item $G$ is a proper MPTG.

\item $G$ is a unit MPTG.
\end{enumerate} 
\end{prop}

\begin{proof}
$(1)\Rightarrow (2)$ Let $G=(V,E)$ be a proper MPTG with representation $\{(I_{v},p_{v})|v\in V\}$ where $I_{v}=[a_{v},b_{v}]$ is an interval and $p_{v}$ is a point within $I_{v}$. From Theorem \ref{pmptg1} it follows that $G$ must possesses a proper MPTG ordering of $V$. Hence one can able to construct the sequences $P_{1},P_{2},P$ as described in Definition \ref{canseq}.
\noindent Let $P_{1}:b_{k_{1}}\prec\hdots\prec b_{k_{n}}$ and $P_{2}:a_{k_{1}}\prec\hdots\prec a_{k_{n}}$. We assign $a_{k_{1}},b_{k_{1}}$ with real values $\alpha_{1},\alpha_{1}+l$ in $P$. At $i$ th step we assign $a_{k_{i}},b_{k_{i}}$ with real values $\alpha_{i},\alpha_{i}+l$ where $\text {max}\{\alpha_{i-1},m_{i}\}<\alpha_{i}<M_{i}$ where $m_{i}=\text{max}\{\alpha_{j}+l \hspace{0.3em}| \hspace{0.3em}1\leq j\leq i-1, b_{k_{j}}\prec a_{k_{i}} \text{in} \hspace{0.3em} P\}$ and $M_{i}=\text{min}\{\alpha_{j}+l \hspace{0.3em}| \hspace{0.3em}1\leq j\leq i-1, a_{k_{i}}\prec b_{k_{j}} \text{in} \hspace{0.3em} P\}$. 
 After the assignment is done for all $a_{k_{i}},b_{k_{i}}$ where $1\leq i\leq n$, assign values for $p_{i}$ in $P$ in such a way so that $P$ still remains an increasing sequence. Now it is easy to check $([a_{k_{i}},b_{k_{i}}],p_{i})$ gives us unit MPTG representation (unit length $l$) as order of $a_{i},b_{i},p_{i}$ remain intact in $P$
after assigning their values in above way. 
\vspace{0.1em}

\noindent $(2)\Rightarrow (1)$ Converse is obvious.
\end{proof}

\begin{rmk}
The unit vs. proper question is addressed for many graph classes in reference \cite{Golumbic2}. Golumbic and Trenk wrote ``For min-tolerance graphs, the classes are different. For max-tolerance graphs, the question is still open'' [\cite{Golumbic2}, p. $215$]. We settle this query below.
\end{rmk}

\begin{prop}\label{unitmptg}
Unit max-tolerance graphs form a strict subclass of proper max-tolerance graphs. 
\end{prop}
\begin{proof}
An unit max-tolerance graph is a proper max-tolerance graphs with the same tolerance representation.
Let $\{v_{1},\hdots,v_{6}\}$ are vertices of $\overline{C_{6}}$ occurred in circularly consecutive order in $C_{6}$. Then $I_{1}=[2,8], t_{1}=2.9, I_{2}=[4,10], t_{2}=4.5, I_{3}=[1,7.1], t_{3}=1, I_{4}=[5,11], t_{4}=3, I_{5}=[3,9], t_{5}=4.1, I_{6}=[5.2,12], t_{6}=1.5$ is a proper max-tolerance representation of $\overline{C_{6}}$. But from \cite{san} one can verify that $\overline{C_{6}}$ is not a unit max-tolerance graph. Hence the result follows.
\end{proof}

\noindent In the following Proposition we will show that proper MPTG and proper max-tolerance graphs are not comparable although both of them belong to the class of MPTG.

\begin{prop}\label{pmtg2}
Proper MPTG and the class of proper max-tolerance graphs are not comparable.
\end{prop}

\begin{proof}
 $C_{5}$ is a proper max-tolerance graph with representation $I_{1}=[1,6],t_{1}=0.25, I_{2}=[1.2,8], t_{2}=4.7, I_{3}=[3,10], t_{3}=4.8, I_{4}=[5,12], t_{4}=4, I_{5}=[5.5,13], t_{5}=0.35$. But $C_{5}\notin \text{proper MPTG}$ follows from Proposition \ref{Cn}.

\noindent Again we note from Proposition \ref{bipt} that $K_{2,3}$ is a proper MPTG. But it is not a proper max-tolerance graph follows from Corollary \ref{k23pf}.
\end{proof}

                                                                                                                                                                                                                                                                                                                                                                                                                                                                                                                                                                                                                                                                                                                                                                                                                                                                                                                                                                                                                                                                                                                                                                                                                                                                                                                                                                                                                                                                                                                                                                                                                                                                                                                                                                                                                                                                                                                                                                                                                                                                                                                                                                                                                                                                                                                                                                                                                                                                                                                                                                                                                                                                                                                                                                                                                                                                                                                                                                                                                                                                                                                                                                                                                                                                                                                                                                                                                                                                                                                                                                                                                                                                                                                                                                                                                                                                                                                                                                                                                                                                                                                                                                                                                                                                \section{Conclusion}

\noindent In this article we introduce proper max-point-tolerance graphs and characterized this graph class in Theorem \ref{pmptg1}. We have proved in previous section that interval graphs $\subset$ proper MPTG. Interestingly one can note that both of these graph classes are AT-free and perfect 
(section \ref{pmpt}, \cite{Lekker}). Therefore any proper MPTG which is not an interval graph must contain $C_{4}$ as an induced subgraph. Also 
 interval graphs $\subset$ central MPTG$=$ unit max-tolerance graph $\subset$ max-tolerance graph according to the reference \cite{san}.
Again Proposition \ref{unitmptg} helps us to conclude unit max-tolerance graphs $\subset$ proper max-tolerance graphs. Next we find that $\overline{C_{6}}\in$ proper MPTG$\setminus$ central MPTG. 
It is a proper MPTG with representation $I_{1}=[20,40],p_{1}=39, I_{2}=[15,38],p_{2}=30,I_{3}=[32,46],p_{3}=33,I_{4}=[25,42],p_{4}=27,I_{5}=[28,44],p_{5}=37,I_{6}=[10,36],p_{6}=34$. But it is not a central MPTG follows from Theorem $3.9$ of \cite{san}.
 Again $C_{n},n\geq 5\in$ central MPTG $\setminus$ proper MPTG follows from \cite{Soto} and Proposition \ref{Cn} of this article. Hence these two classes become incomparable.  Next in Proposition \ref{pmtg2} we see that proper max-tolerance graphs are not comparable to proper MPTG whereas both of these graph classes belong to the class of MPTG. Next we find $G_{1}$ in Figure \ref{atfree1} is a MPTG. It is also max-tolerance graph with representation $I_{1}=[10,25],t_{1}=5, I_{2}=[45,53],t_{2}=8, I_{3}=[65,75],t_{3}=10,I_{4}=[20,40],t_{4}=5,I_{5}=[30,70],t_{5}=10,I_{6}=[45,60],t_{6}=8,I_{7}=[60,80],t_{7}=10$. But it is not a proper MPTG. 
Again $K_{2,3}$ is a max-tolerance graph with representation $I_{x}=[-20,0], t_{x}=1, I_{y}=[0,20], t_{y}=1$ for one partite set and $I_{v_{1}}=[-2,2], t_{v_{1}}=1, I_{v_{2}}=[-6,6], t_{v_{2}}=5, I_{v_{3}}=[-20,20], t_{v_{3}}=19$ for the other partite set. But it is not a proper max-tolerance graph from Proposition \ref{k23pf}.
Lastly from Theorem \ref{incom} we establish MPTG and max tolerance graphs are not comparable. Combining these we obtain relations between some subclasses of max-tolerance graphs and MPTG related to proper MPTG in Figure \ref{fig}.

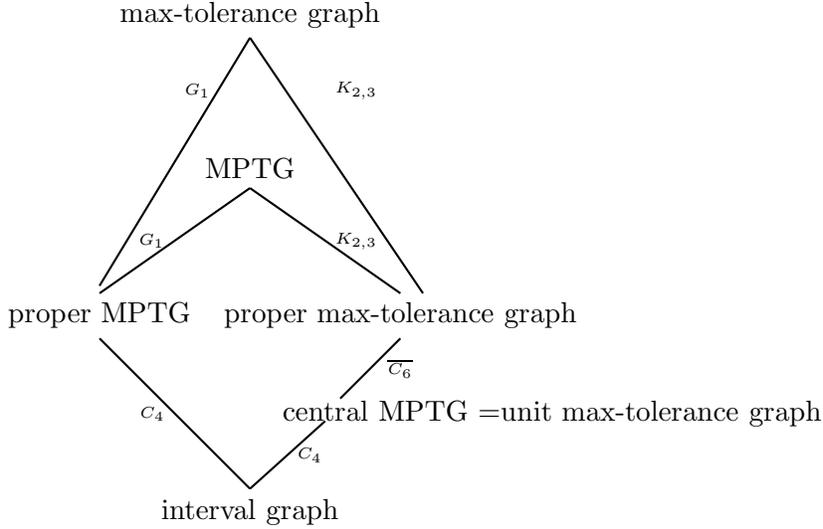
\begin{figure}
\centering
\begin{tikzpicture}
\draw[-][draw=black, thick] (0,0) -- (-2,2);
\draw[-][draw=black, thick] (0,0) -- (1,0.9);
\draw[-][draw=black, thick] (1.2,1.2) -- (2,2);
\draw[-][draw=black, thick] (-2,2.6) -- (0,4);
\draw[-][draw=black, thick] (2,2.6) -- (0,4);
\draw[-][draw=black, thick] (2.3,2.6) -- (0,6);
\draw[-][draw=black, thick] (-2,2.7) -- (0,6);

\node [below] at (0,0) {{interval graph}};
\node [above] at (-2,2) {{proper MPTG}};
\node [right] at (0.3,1) {{central MPTG $=$unit max-tolerance graph}};
\node [above] at (2,2) {{proper max-tolerance graph}};
\node [above] at (0,4) {{MPTG}};
\node [above] at (0,6) {{max-tolerance graph}};

\node [left] at (-1,1) {{\tiny{$C_{4}$}}};
\node [right] at (0.5,0.45) {{\tiny{$C_{4}$}}};
\node [right] at (1.7,1.6) {{\tiny{$\overline{C_{6}}$}}};
\node [right] at (1,5.3) {{\tiny{$K_{2,3}$}}};
\node [right] at (1,3.3) {{\tiny{$K_{2,3}$}}};
\node [left] at (-1,3.3) {{\tiny{$G_{1}$}}};
\node [right] at (-1,5.3) {{\tiny{$G_{1}$}}};
\end{tikzpicture}
\caption{Hierarchy between subclasses of max tolerance graph and MPTG}\label{fig}
\end{figure}

\vspace{0.6em}
\noindent
\textbf{Acknowledgements:}

\vspace{.2cm}
\noindent This research is supported by UGC (University Grants Commission) NET fellowship (21/12/2014(ii)EU-V) of the author.

\vspace{0.6em}
\noindent \textbf{Appendix}
\vspace{0.5em}

\noindent Here we provide the detailed verification of the other cases occurred in Lemma \ref{lm2}. 
\vspace{0.3em}

\noindent \textbf{Case (ii) } $i_{2}<i_{3}<i_{1}$

\noindent As $b_{i_{1}}\prec b_{i_{2}}$, if $b_{i_{1}} R_{1} b_{i_{2}}$ then there exists some $k_{2}>i_{1}$ such that $a_{i_{1},k_{2}}=0, a_{i_{2},k_{2}}=1$. We note that $a_{i_{3},k_{2}}$ must be one otherwise from (\ref{1}) we get $b_{i_{3}} R_{1} b_{i_{2}}$, i.e; $b_{i_{3}}\prec b_{i_{2}}$ which is contradicts our assumption. Hence $b_{i_{1}} R_{1} b_{i_{3}}$ from (\ref{1}), i.e; $b_{i_{1}}\prec b_{i_{3}}$. Now if $b_{i_{1}} R_{2} b_{i_{2}}$ then there exists some $k_{1}<i_{2}$ such that $a_{k_{1},i_{2}}=0, a_{k_{1},i_{1}}=1$. Now if $a_{k_{1},i_{3}}=1$ then $b_{i_{3}} R_{2} b_{i_{2}}$ from (\ref{2}) imply $b_{i_{3}}\prec b_{i_{2}}$ which is again a contradiction. Hence $a_{k_{1},i_{3}}=0$, which imply $b_{i_{1}} R_{2} b_{i_{3}}$ from (\ref{2}), i.e; $b_{i_{1}}\prec b_{i_{3}}$.
Now if $b_{i_{1}} R_{1} b_{k}$ and $b_{k} R_{2} b_{i_{2}}$ for some $k\in V$ such that $i_{2}<k<i_{1}$. Then there exists $k_{2}>i_{1}$ and $k_{1}<i_{2}$ such that $a_{i_{1},k_{2}}=0, a_{k,k_{2}}=1, a_{k_{1},i_{2}}=0, a_{k_{1},k}=1$. Now when $i_{2}<k<i_{3}$, if $a_{i_{3},k_{2}}=0$ then $b_{i_{3}} R_{1} b_{k}$ imply $b_{i_{3}}\prec b_{i_{2}}$ from (\ref{3}) as $b_{k} R_{2} b_{i_{2}}$ which is a contradiction. Hence $a_{i_{3},k_{2}}=1$ and from (\ref{1}) it follows  $b_{i_{1}} R_{1} b_{i_{3}}$, i.e; $b_{i_{1}}\prec b_{i_{3}}$. 
Now when $i_{3}<k<i_{1}$, if $a_{k_{1},i_{3}}=1$ then $b_{i_{3}} R_{2} b_{i_{2}}$ from (\ref{2}) imply $b_{i_{3}}\prec b_{i_{2}}$ which is a contradiction. Hence $a_{k_{1},i_{3}}=0$  which imply $b_{k} R_{2} b_{i_{3}}$ which helps us to conclude $b_{i_{1}}\prec b_{i_{3}}$ from (\ref{3}) as $b_{i_{1}} R_{1} b_{k}$.

\vspace{0.3em}
\noindent \textbf{Case (iii)}  $i_{1}<i_{2}<i_{3}$ 

\noindent In this case we assume on contrary $b_{i_{3}}\prec b_{i_{1}}$.

\noindent Let $b_{i_{3}} R_{1} b_{i_{1}}$. Then there exists $k_{2}>i_{3}$ such that $a_{i_{3},k_{2}}=0, a_{i_{1},k_{2}}=1$. Now if $a_{i_{2},k_{2}}=1$ then $b_{i_{3}} R_{1} b_{i_{2}}$ from (\ref{1}) imply $b_{i_{3}}\prec b_{i_{2}}$ which contradicts our assumption. Again if $a_{i_{2},k_{2}}=0$ then $b_{i_{2}} R_{1} b_{i_{1}}$ imply $b_{i_{2}}\prec b_{i_{1}}$ which is again a contradiction. Now if $b_{i_{3}} R_{2} b_{i_{1}}$ then there exists some $k_{1}<i_{1}$ such that $a_{k_{1},i_{1}}=0, a_{k_{1},i_{3}}=1$. Now if $a_{k_{1},i_{2}}=0$ then $b_{i_{3}} R_{2} b_{i_{2}}$ from (\ref{2}) imply $b_{i_{3}}\prec b_{i_{2}}$ which is a contradiction.
Again if $a_{k_{1},i_{2}}=1$ then $b_{i_{2}} R_{2} b_{i_{1}}$ imply $b_{i_{2}}\prec b_{i_{1}}$ which is again a contradiction. Now let $b_{i_{3}} R_{1} b_{k}$ and $b_{k} R_{2} b_{i_{1}}$ for some $k\in V$ such that $i_{1}<k<i_{3}$. Then there exist $k_{2}>i_{3}$ and $k_{1}<i_{1}$ such that $a_{i_{3},k_{2}}=0, a_{k,k_{2}}=1, a_{k_{1},i_{1}}=0, a_{k_{1},k}=1$. Now when $i_{1}<k<i_{2}$ then if $a_{i_{2},k_{2}}=1$ then $b_{i_{3}} R_{1} b_{i_{2}}$ from (\ref{1}) imply $b_{i_{3}}\prec b_{i_{2}}$ which is not true. Again if $a_{i_{2},k_{2}}=0$ then $ b_{i_{2}} R_{1} b_{k}$ from (\ref{1}) which helps us to conclude $b_{i_{2}}\prec b_{i_{1}}$ from (\ref{3}) as $b_{k} R_{2} b_{i_{1}}$ which is again a contradiction. Now when $i_{2}<k<i_{3}$ then if  $a_{k_{1},i_{2}}=1$, $b_{i_{2}} R_{2} b_{i_{1}}$, i.e; $b_{i_{2}}\prec b_{i_{1}}$ from (\ref{2}) which is not true. Again if $a_{k_{1},i_{2}}=0$ then $b_{k} R_{2} b_{i_{2}}$ which helps us to conclude $b_{i_{3}}\prec b_{i_{2}}$  from (\ref{3}) as $b_{i_{3}} R_{1} b_{k}$ which is again a contradiction. 

\vspace{0.3em}
\noindent \textbf{Case (iv)} $i_{3}<i_{1}<i_{2}$

\noindent As $b_{i_{2}}\prec b_{i_{3}}$, if $b_{i_{2}} R_{1} b_{i_{3}}$ then there exist $k_{2}>i_{2}$ such that $a_{i_{2},k_{2}}=0, a_{i_{3},k_{2}}=1$. Now if $a_{i_{1},k_{2}}=1$ then $b_{i_{2}} R_{1} b_{i_{1}}$ from (\ref{1}), i.e; $b_{i_{2}}\prec b_{i_{1}}$ which is a contradiction. Hence $a_{i_{1},k_{2}}=0$ which imply $b_{i_{1}} R_{1} b_{i_{3}}$. Therefore $b_{i_{1}}\prec b_{i_{3}}$. If $b_{i_{2}} R_{2} b_{i_{3}}$ then there exist $k_{1}<i_{3}$ such that $a_{k_{1},i_{3}}=0, a_{k_{1},i_{2}}=1$. If $a_{k_{1},i_{1}}=0$ then $b_{i_{2}} R_{2} b_{i_{1}}$from (\ref{2}) imply
$b_{i_{2}}\prec b_{i_{1}}$ which is a contradiction. Hence $a_{k_{1},i_{1}}=1$ which imply $b_{i_{1}} R_{2} b_{i_{3}}$, i.e; $b_{i_{1}}\prec b_{i_{3}}$. 
If $b_{i_{2}} R_{1} b_{k}$ and $b_{k} R_{2} b_{i_{3}}$ for some $k\in V$ such that $i_{3}<k<i_{2}$. Then there exist $k_{2}>i_{2},k_{1}<i_{3}$ such that $a_{i_{2},k_{2}}=0, a_{k,k_{2}}=1, a_{k_{1},i_{3}}=0, a_{k_{1},k}=1$. Now when $i_{3}<k<i_{1}$, 
if $a_{i_{1},k_{2}}=1 $ then $b_{i_{2}} R_{1} b_{i_{1}}$
 from (\ref{1}) which imply $b_{i_{2}}\prec b_{i_{1}}$ which is a contradiction. Hence $a_{i_{1},k_{2}}=0$ which imply $b_{i_{1}} R_{1} b_{k}$ which helps us to conclude $b_{i_{1}}\prec b_{i_{3}}$ from (\ref{3}) as $b_{k} R_{2} b_{i_{3}}$. Again when $i_{1}<k<i_{2}$, if
 $a_{k_{1},i_{1}}=0$ then $b_{k} R_{2} b_{i_{1}}$ imply $b_{i_{2}}\prec b_{i_{1}}$ from (\ref{3}) (as $b_{i_{2}} R_{1} b_{k}$) which contradicts our assumption. Hence $a_{k_{1},i_{1}}=1$ which imply $b_{i_{1}}\prec b_{i_{3}}$ as earlier.  

\vspace{0.3em}
\noindent \textbf{Case (v)} $i_{3}<i_{2}<i_{1}$

\noindent We need the following observations to prove this case.

\noindent \textit{Observations:}
\begin{enumerate}

\item First note $a_{i_{3},i_{2}}=a_{i_{2},i_{1}}=1$ as $i_{3}<i_{2}$ and $i_{2}<i_{1}$.

\item Next if $b_{i_{1}} R_{1} b_{r}$ for any $r<i_{1}$ then there exist some $\bf{k_{2}}>i_{1}$ such that $a_{i_{1},k_{2}}=0, a_{r,k_{2}}=1$.

\item Next $b_{m} R_{2} b_{i_{2}}$ for any $m>i_{2}$ imply existence of some $\bf{k_{1}}<i_{2}$ such that $a_{k_{1},i_{2}}=0, a_{k_{1},m}=1$. Now if $i_{3}<k_{1}<i_{2}$ then $a_{i_{3},i_{2}}$ becomes zero as $a_{k_{1},i_{2}}=0$ is up open which imply $b_{i_{3}}\prec b_{i_{2}}$ which contradicts our assumption. Hence $k_{1}<i_{3}$.

\item Again $b_{i_{2}} R_{1} b_{j}$ for any $j<i_{2}$ ensure existence of some $\bf{k_{2}^{'}}>i_{2}$ satisfying $a_{i_{2},k_{2}^{'}}=0, a_{j,k_{2}^{'}}=1$. Clearly $k_{2}^{'}>i_{1}$ otherwise (\ref{4p1}) gets contradicted for vertices $\{j,i_{2},k_{2}^{'},i_{1}\}$.

\item Similarly $b_{l} R_{2} b_{i_{3}}$ for any $l>i_{3}$ ensure existence of some $\bf{k_{1}^{'}}<i_{3}$ satisfying $a_{k_{1}^{'},i_{3}}=0, a_{k_{1}^{'},l}=1$.

\item Let $b_{i_{2}} R_{1} b_{i_{3}}$. Then $k_{2}^{'}>i_{1}$ from observation $4$. Now if $b_{p} R_{2} b_{i_{2}}$ for any $p>i_{2}$ then $k_{1}<i_{3}$ from observation $3$. Next if $a_{k_{1},i_{3}}=1$ then
condition \ref{5pt1} gets contradicted for vertices $\{k_{1},i_{3},i_{2},p,k_{2}^{'}\}$. Hence $a_{k_{1},i_{3}}=0$, which imply $b_{p} R_{2} b_{i_{3}}$ from (\ref{2}).

\item Let $b_{i_{2}} R_{2} b_{i_{3}}$. Then $k_{1}^{'}<i_{3}$ from observation $5$. Now if $b_{s} R_{2} b_{i_{2}}$ for any $s>i_{2}$ then $k_{1}<i_{3}$ from observation $3$. Next if $k_{1}^{'}<k_{1}$ then (\ref{4p1}) gets contradicted for vertices $\{k_{1}^{'},k_{1},i_{2},s\}$. 
Hence $k_{1}<k_{1}^{'}<i_{3}$ which imply $a_{k_{1},i_{3}}=0$ as $a_{k_{1}^{'},i_{3}}=0$ is up open. Thus we get $b_{s} R_{2} b_{i_{3}}$ from (\ref{2}). 
\end{enumerate}

\noindent \textit{Main proof:} Let $b_{i_{1}} R_{1} b_{i_{2}}$. Now if $b_{i_{2}} R_{1} b_{j}$ for any $j<i_{2}$ then using observations $2$ and $4$ we note  $k_{2}^{'}>k_{2}$ otherwise (\ref{4p1}) gets contradicted for the vertices $\{j,i_{2},k_{2}^{'},k_{2}\}$. Now $a_{i_{1},k_{2}^{'}}=0$ clearly as $a_{i_{1},k_{2}}=0$ is right open. This imply $b_{i_{1}} R_{1} b_{j}$ from (\ref{1}) which imply $b_{i_{1}}\prec b_{j}$. 
\vspace{0.3em}

\noindent Now if $b_{i_{2}} R_{1} b_{i_{3}}$, we get $b_{i_{1}}\prec b_{i_{3}}$ from last paragraph. Again if $b_{i_{2}} R_{2} b_{i_{3}}$ then $b_{i_{1}}\prec b_{i_{3}}$ from (\ref{3}) as $b_{i_{1}} R_{1} b_{i_{2}}$ and $i_{3}<i_{2}<i_{1}$. When $b_{i_{2}} R_{1} b_{k}$ and $b_{k} R_{2} b_{i_{3}}$ for some $k\in V$ such that $i_{3}<k<i_{2}$, we get $b_{i_{1}} R_{1} b_{k}$ from above paragraph. Hence we can  conclude $b_{i_{1}}\prec b_{i_{3}}$ from (\ref{3}) as $b_{k} R_{2} b_{i_{3}}$. 

\vspace{0.3em}
\noindent Now let $b_{i_{1}} R_{2} b_{i_{2}}$. If $b_{i_{2}} R_{1} b_{i_{3}}$ then using observations $3,4,6$ one can get $b_{i_{1}} R_{2} b_{i_{3}}$, i.e; $b_{i_{1}}\prec b_{i_{3}}$.  Again when $b_{i_{2}} R_{2} b_{i_{3}}$ then $b_{i_{1}} R_{2} b_{i_{3}}$ follows from observation $7$, i.e; $b_{i_{1}}\prec b_{i_{3}}$.
Now if $b_{i_{2}} R_{1} b_{k}$ and $b_{k} R_{2} b_{i_{3}}$ for some $k\in V$ such that $i_{3}<k<i_{2}$ then if $a_{k_{1},k}=1$, condition \ref{5pt1} gets contradicted for vertices $\{k_{1},k,i_{2},i_{1},k_{2}^{'}\}$ (observations $3,4$). Hence $a_{k_{1},k}=0$ which imply $b_{i_{1}} R_{2} b_{k}$ which helps us to conclude $b_{i_{1}}\prec b_{i_{3}}$ from (\ref{3}) as $b_{k} R_{2} b_{i_{3}}$.

 \vspace{0.3em}
\noindent Now let $b_{i_{1}} R_{1} b_{k}$ and $b_{k} R_{2} b_{i_{2}}$ for some $k\in V$ such that $i_{2}<k<i_{1}$. 
 If $b_{i_{2}} R_{1} b_{i_{3}}$ then from observations $3,4,6$ we get $b_{k} R_{2} b_{i_{3}}$ which helps us to conclude $b_{i_{1}}\prec b_{i_{3}}$ from (\ref{3}) as $b_{i_{1}} R_{1} b_{k}$.
Again if $b_{i_{2}} R_{2} b_{i_{3}}$ then $b_{k} R_{2} b_{i_{3}}$ from
observation $7$ which helps us to conclude $b_{i_{1}}\prec b_{i_{3}}$ from (\ref{3}) as $b_{i_{1}} R_{1} b_{k}$.

\vspace{0.3em}
\noindent Now if $b_{i_{2}} R_{1} b_{k^{'}}$ and $b_{k^{'}} R_{2} b_{i_{3}}$ for some $k^{'}\in V$ such that $i_{3}<k^{'}<i_{2}$. Then if $k_{1}^{'}<k_{1}$ (observations $3,5$) then $a_{k_{1},k^{'}}=1$ applying (\ref{4p1}) on vertices $\{k_{1}^{'},k_{1},k^{'},k\}$. Therefore condition \ref{5pt1} of Lemma 
\ref{lm1} gets contradicted for vertices $\{k_{1},k^{'},i_{2},k,k_{2}^{'}\}$ (observations $3,4$). Hence we get $k_{1}<k_{1}^{'}<i_{3}$. Thus $a_{k_{1},i_{3}}$ becomes zero as $a_{k_{1}^{'},i_{3}}=0$ is up open, which imply $b_{k} R_{2} b_{i_{3}}$ from (\ref{2}). Hence from (\ref{3}) we get $b_{i_{1}}\prec b_{i_{3}}$ as $b_{i_{1}} R_{1} b_{k}$ and $i_{3}<k<i_{1}$.
 
\vspace{0.3em}
 
\noindent \textbf{Case(vi)} $i_{2}<i_{1}<i_{3}$ 

\noindent Lets assume $b_{i_{3}}\prec b_{i_{1}}$ on contrary. We do the following observations which we will use throughout the proof.

\noindent \textit{Observations:}
\begin{enumerate}

\item First note that $a_{i_{2},i_{1}}=1$ as $i_{2}<i_{1}$.

\item Next if $b_{i_{1}} R_{1} b_{r}$ for any $r<i_{1}$ then there exists $\bf{k_{2}}>i_{1}$ such that $a_{i_{1},k_{2}}=0, a_{r,k_{2}}=1$. Now if $i_{1}<k_{2}<i_{3}$ then $a_{i_{1},i_{3}}$ becomes zero as $a_{i_{1},k_{2}}=0$ is right open, which imply $b_{i_{1}}\prec b_{i_{3}}$. But this contradicts our assumption. Hence $k_{2}>i_{3}$.

\item Again if $b_{m} R_{2} b_{i_{2}}$ for any $m>i_{2}$ then there exists $\bf{k_{1}}<i_{2}$ satisfying $a_{k_{1},i_{2}}=0, a_{k_{1},m}=1$.

\item If $b_{l} R_{2} b_{i_{1}}$ for any $l>i_{1}$ then there exists $\bf{k_{1}^{'}}<i_{1}$ satisfying $a_{k_{1}^{'},i_{1}}=0, a_{k_{1}^{'},l}=1$. Now if $i_{2}<k_{1}^{'}<i_{1}$ then (\ref{4p1}) gets contradicted for vertices $\{i_{2},k_{1}^{'},i_{1},l\}$ using observation $1$. Hence $k_{1}^{'}<i_{2}$.

\item Again if $b_{i_{3}} R_{1} b_{j}$ for any $j<i_{3}$ ensure existence of some $\bf{k_{2}^{'}}>i_{3}$ such that $a_{i_{3},k_{2}^{'}}=0, a_{j,k_{2}^{'}}=1$.

\item Let $b_{i_{1}} R_{1} b_{i_{2}}$. Then  $k_{2}>i_{3}$ from observation $2$. Now if $b_{j} R_{2} b_{i_{1}}$ for any $i_{1}<j<k_{2}$ then from observation $4$ we get $k_{1}^{'}<i_{2}$.   Again if $a_{k_{1}^{'},i_{2}}=1$ then condition \ref{5pt1} gets contradicted for vertices $\{k_{1}^{'},i_{2},i_{1},j,k_{2}\}$. Hence $\bf{a_{k_{1}^{'},i_{2}}=0}$.

\item Let $b_{i_{1}} R_{2} b_{i_{2}}$.  Then $k_{1}<i_{2}$ from observation $3$. Now if $b_{l} R_{2} b_{i_{1}}$ holds for any $l>i_{1}$ then $k_{1}^{'}<i_{2}$ from observation $4$. Moreover $\bf{k_{1}^{'}<k_{1}}$ otherwise (\ref{4p1}) gets contradicted for vertices $\{k_{1},k_{1}^{'},i_{1},l\}$.

\end{enumerate}

\noindent \textit{Main proof:}

\noindent Let $b_{i_{1}} R_{1} b_{i_{2}}$. Now if $b_{i_{3}} R_{1} b_{i_{1}}$ holds then if $k_{2}^{'}>k_{2}$ then (\ref{4p1}) gets contradicted for the vertices $\{i_{2},i_{1},k_{2},k_{2}^{'}\}$ (observations $2,5$). Again if $i_{3}<k_{2}^{'}<k_{2}$ then $a_{i_{3},k_{2}}$ becomes zero as $a_{i_{3},k_{2}^{'}}=0$ is right open, which imply $b_{i_{3}} R_{1} b_{i_{2}}$ from (\ref{1}), i.e; $b_{i_{3}}\prec b_{i_{2}}$ which is a contradiction. Now if $b_{i_{3}} R_{2} b_{i_{1}}$ then  $k_{1}^{'}<i_{2}$ from observation $4$. Again $a_{k_{{1}^{'}},i_{2}}=0$ from observation $6$. But in this case $b_{i_{3}} R_{2} b_{i_{2}}$ from (\ref{2}) imply $b_{i_{3}}\prec b_{i_{2}}$ which is not true.
Now if $b_{i_{3}} R_{1} b_{k}$ and $b_{k} R_{2} b_{i_{1}}$ for some $k\in V$ such that $i_{1}<k<i_{3}$. Note that $k_{1}^{'}<i_{2}$
from observation $4$. Again $a_{k_{1}^{'},i_{2}}=0$ from observation $6$.
Hence we get $b_{k} R_{2} b_{i_{2}}$ which along with $b_{i_{3}} R_{1} b_{k}$ helps us to conclude  $b_{i_{3}}\prec b_{i_{2}}$ from (\ref{3}), which is not true.
\vspace{0.3em}

\noindent Now let $b_{i_{1}} R_{2} b_{i_{2}}$. 
Now if $b_{i_{3}} R_{1} b_{i_{1}}$ then $b_{i_{3}}\prec b_{i_{2}}$ from (\ref{3}) as $b_{i_{1}} R_{2} b_{i_{2}}$, which is a contradiction. 
If $b_{i_{3}} R_{2} b_{i_{1}}$ then $k_{1}^{'}<k_{1}$ from observations $3,4,7$. Hence $a_{k_{1}^{'},i_{2}}=0$ as $a_{k_{1},i_{2}}=0$ is up open, which imply $b_{i_{3}} R_{2} b_{i_{2}}$ from (\ref{2}), i.e; $b_{i_{3}}\prec b_{i_{2}}$ which is a contradiction. Again if $b_{i_{3}} R_{1} b_{k}$ and $b_{k} R_{2} b_{i_{1}}$ for some $k\in V$ such that $i_{1}<k<i_{3}$ then $k_{2}^{'}>i_{3}$ and $k_{1}^{'}<i_{1}$ follows from  observations $4,5$. Again note $k_{1}^{'}<k_{1}$ from observation $7$.
Now $a_{k_{1}^{'},i_{2}}$ must be zero as $a_{k_{1},i_{2}}=0$ is up open, which imply $b_{k} R_{2} b_{i_{2}}$ from (\ref{2}) which helps us to conclude $b_{i_{3}}\prec b_{i_{2}}$ from (\ref{3}) as $ b_{i_{3}} R_{1} b_{k}$, which again contradicts our assumption. 
\vspace{0.3em}

\noindent Now let $b_{i_{1}} R_{1} b_{k}$ and $b_{k} R_{2} b_{i_{2}}$ for some $k\in V$ such that $i_{2}<k<i_{1}$. Then $k_{2}>i_{3}$ from observation $2$. 
 Now if $b_{i_{3}} R_{1} b_{i_{1}}$ then $k_{2}^{'}$ (observation $5$) can not be greater than $k_{2}$ as in that case (\ref{4p1}) gets contradicted for vertices $\{k,i_{1},k_{2},k_{2}^{'}\}$.
Again if $i_{3}<k_{2}^{'}<k_{2}$ then $a_{i_{3},k_{2}}=0$ as $a_{i_{3},k_{2}^{'}}=0$ is right open, which imply $b_{i_{3}} R_{1} b_{k}$ from (\ref{1}) which helps us to conclude $b_{i_{3}}\prec b_{i_{2}}$ from (\ref{3}) as $b_{k} R_{2} b_{i_{2}}$, which is again a contradiction.

\vspace{0.3 em}
\noindent Now if $b_{i_{3}} R_{2} b_{i_{1}}$ then $k_{1}^{'}<i_{2}$ from observation $4$.  
 Now if $k_{1}<k_{1}^{'}<i_{2}$ (observations $3,4$) then $a_{k_{1}^{'},k}=1$ applying (\ref{4p1}) on vertices $\{k_{1},k_{1}^{'},k,i_{3}\}$. Now applying condition \ref{5pt1} on vertices $\{k_{1}^{'},k,i_{1},i_{3},k_{2}\}$ one can find a contradiction. Now if $k_{1}^{'}<k_{1}$ then $a_{k_{1}^{'},i_{2}}$ becomes zero as $a_{k_{1},i_{2}}=0$ is up open, which imply $b_{i_{3}} R_{2} b_{i_{2}}$ from (\ref{2}), i.e, $b_{i_{3}}\prec b_{i_{2}}$ which is a contradiction.
 
\vspace{0.3em}

\noindent Again  if $b_{i_{3}} R_{1} b_{k}^{'} $ and $b_{k}^{'} R_{2} b_{i_{1}}$ for some $k^{'}\in V$ such that $i_{1}<k^{'}<i_{3}$ then $k_{1}^{'}<i_{2}$ from observation $4$. Now if $k_{1}<k_{1}^{'}<i_{2}$ then $a_{k_{1}^{'},k}$ become one from (\ref{4p1}) applying or vertices $\{k_{1},k_{1}^{'},k,k^{'}\}$. Hence applying condition \ref{5pt1} on vertices $\{k_{1}^{'},k,i_{1},k^{'},k_{2}\}$ one can get a contradiction. Now if $k_{1}^{'}<k_{1}$ then $a_{k_{1}^{'},i_{2}}$ becomes zero as $a_{k_{1},i_{2}}=0$ is up open, which imply $b_{k^{'}} R_{2} b_{i_{2}}$ from (\ref{2}) which helps us to conclude $b_{i_{3}}\prec b_{i_{2}}$ from (\ref{3}) as $b_{i_{3}} R_{1} b_{k^{'}}$, which is again a contradiction. 
\end{document}